\documentclass[11pt]{article}
\usepackage{graphicx}

\usepackage{amsmath}
\usepackage{amssymb}

\usepackage{color}
\usepackage[x11names,usenames,dvipsnames,svgnames,table,rgb]{xcolor}
\usepackage{subeqnarray}
\usepackage{soul}

\usepackage{amsthm}
\newtheorem{theorem}{Theorem}[section]
\newtheorem{lemma}[theorem]{Lemma}
\newtheorem{proposition}[theorem]{Proposition}

\theoremstyle{remark}

\theoremstyle{definition}

\renewcommand{\vec}[1]{\boldsymbol{#1}}
\newcommand{\R}{\mathbb{R}}
\newcommand{\N}{\mathbb{N}}
\newcommand{\E}{\mathbb{E}}

\usepackage{todonotes}

\newcommand{\tth}{\vec{\theta}}
\newcommand{\zzh}{\vec{\zeta}}
\newcommand{\eeps}{\vec{\epsilon}}
\renewcommand{\aa}{\vec{a}}
\newcommand{\yy}{\vec{y}}
\newcommand{\xx}{\vec{x}}
\newcommand{\ww}{\vec{w}}
\newcommand{\nn}{\vec{n}}

\newcommand{\uu}{\vec{u}}
\newcommand{\vv}{\vec{v}}
\renewcommand{\E}{\mathbb{E}}
\newcommand{\V}{\mathbb{V}}

\newcommand{\Diag}{\operatorname{Diag}}
\newcommand{\trace}{\operatorname{trace}}

\newcommand{\FB}{\texttt{FB}}
\newcommand{\FISTA}{\texttt{FISTA}}

\newcommand{\VDM}{\texttt{VDM}}
\newcommand{\MUL}{\texttt{MUL}}
\newcommand{\ABCDcy}{\texttt{ABCD-cy}}
\newcommand{\ABCDrp}{\texttt{ABCD-rp}}

\usepackage{algorithm}
\PassOptionsToPackage{noend}{algpseudocode}
\usepackage{algpseudocode}
\algnewcommand\algorithmicinput{\textbf{Input:}}
\algnewcommand\alginput{\item[\algorithmicinput]}
\algnewcommand\algorithmicoutput{\textbf{Output:}}
\algnewcommand\algoutput{\item[\algorithmicoutput]}

\errorcontextlines\maxdimen

\makeatletter
    \newcommand*{\algrule}[1][\algorithmicindent]{\makebox[#1][l]{\hspace*{.5em}\thealgruleextra\vrule height \thealgruleheight depth \thealgruledepth}}%
\newcommand*{\thealgruleextra}{}
\newcommand*{\thealgruleheight}{.75\baselineskip}
\newcommand*{\thealgruledepth}{.25\baselineskip}

\newcount\ALG@printindent@tempcnta
\def\ALG@printindent{%
    \ifnum \theALG@nested>0
        \ifx\ALG@text\ALG@x@notext
        \else
            \unskip
            \addvspace{-1pt}
            \ALG@printindent@tempcnta=1
            \loop
                \algrule[\csname ALG@ind@\the\ALG@printindent@tempcnta\endcsname]%
                \advance \ALG@printindent@tempcnta 1
            \ifnum \ALG@printindent@tempcnta<\numexpr\theALG@nested+1\relax
            \repeat
        \fi
    \fi
    }%
\usepackage{etoolbox}
\patchcmd{\ALG@doentity}{\noindent\hskip\ALG@tlm}{\ALG@printindent}{}{\errmessage{failed to patch}}
\makeatother

\newbox\statebox
\newcommand{\myState}[1]{%
    \setbox\statebox=\vbox{#1}%
    \edef\thealgruleheight{\dimexpr \the\ht\statebox+1pt\relax}%
    \edef\thealgruledepth{\dimexpr \the\dp\statebox+1pt\relax}%
    \ifdim\thealgruleheight<.75\baselineskip
        \def\thealgruleheight{\dimexpr .75\baselineskip+1pt\relax}%
    \fi
    \ifdim\thealgruledepth<.25\baselineskip
        \def\thealgruledepth{\dimexpr .25\baselineskip+1pt\relax}%
    \fi
    \State #1%
    \def\thealgruleheight{\dimexpr .75\baselineskip+1pt\relax}%
    \def\thealgruledepth{\dimexpr .25\baselineskip+1pt\relax}%
}

\usepackage{pgfplots}

\begin{document}

\title{An unexpected connection between Bayes $A-$optimal designs and the Group Lasso
}

\author{Guillaume Sagnol\footnote{This work was initiated when the first author was invited in Toulouse by the Chair in Applied Mathematics OQUAIDO, gathering partners in technological research (BRGM, CEA, IFPEN, IRSN, Safran, Storengy) and academia (CNRS, Ecole Centrale de Lyon, Mines Saint-Etienne, University of Grenoble, University of Nice, University of Toulouse) around advanced methods for Computer Experiments.}
\footnote{The research of the first author is carried out in the framework of MATHEON supported by Einstein Foundation Berlin.}
\\ \emph{TU Berlin}
\and Edouard Pauwels \\ \emph{Toulouse 3 Universit\'{e} Paul Sabatier}}



\date{}

\maketitle

\begin{abstract}
				We show that the $A$-optimal design optimization problem over $m$ design points in $\R^n$ is equivalent to minimizing a quadratic function plus a group lasso sparsity inducing term over $n\times m$ real matrices. This observation allows to describe several new algorithms for $A$-optimal design based on splitting and block coordinate decomposition. These techniques are well known and proved powerful to treat large scale problems in machine learning and signal processing communities. The proposed algorithms come with rigorous convergence guaranties and convergence rate estimate stemming from the optimization literature. Performances are illustrated on synthetic benchmarks and compared to existing methods for solving the optimal design problem.

\end{abstract}

\section{Introduction}
We consider an optimal experimental design problem of the form
\begin{equation}\label{theBasicProblem}
  \underset{\vec{w} \in \Delta}{\vec{\operatorname{minimize}}}\ \ \Phi_{A_K} \left( \Sigma^{-1} + \frac{N}{\sigma^2} \sum_{i=1}^m w_i \aa_i \aa_i^T
\right),
\end{equation}
where $\Phi_{A_K}(M)=\trace K^T M^{-1} K$ is the criterion of $A_K$-optimality for some matrix $K\in\R^{n \times r}$ depending on the quantity to be estimated,
$\Sigma$ is a known positive definite matrix, the constants $N$, $\sigma$ and the vectors $\aa_i \in\R^n$, $i=1,\ldots,m,$
are known,
and $\Delta:=\{\vec{w}\in\R^m:\ \vec{w}\geq\vec{0},\ \sum_{i=1}^m w_i=1\}$ is the probability simplex.
This problem arises in linear regression models with
a finite design space, which we identify with  $[m]:=\{1,...,m\}$,  in which independent trials at the $i$th design points yield random measurements $Y_i$, satisfying
$\mathbb{E}[Y_i] = \aa_i^T\vec{\theta}$, $\mathbb{V}[Y_i]=\sigma^2$, for all $i\in[m]$.
In addition, a prior noisy
observation $\tth_0$ of the unknown parameter $\tth\in\R^n$ is available,
with variance-covariance matrix $\mathbb{V}[\tth_0]=\Sigma$ and expectation $\mathbb{E}[\tth_0] = \tth$. 
Then, Problem~\eqref{theBasicProblem} can be interpreted as
selecting the \emph{optimal} fraction $w_i$ of a total number $N$ of trials to perform on the $i$th design point
(the meaning of \emph{optimal} will be detailed in the next section).

This problem was first introduced in~\cite{DD76} under the name \emph{$\psi$-optimality},
and studied in detail by Chaloner~\cite{Cha84}, who observed that this problem could also be called \emph{Bayes A-optimality},
a name still used in the literature. 
Nevertheless, Bayes-optimal designs can also be used in a non-Bayesian context,
when the experimenter is committed to a first batch of trials, and need to select an additional batch of $N$ trials, cf.~\cite{Cha84}.

\medskip
We should observe that
Problem~\eqref{theBasicProblem} is in fact the continuous relaxation of the following discrete problem, which we call \emph{$N$-exact Bayes $A_K$-optimal design}:
\begin{equation}\label{theBasicProblem_discrete}
  \underset{\vec{n}\in\Delta_N}{\vec{\operatorname{minimize}}}\ \ \Phi_{A_K} \left( \Sigma^{-1} + \frac{1}{\sigma^2} \sum_{i=1}^m n_i \aa_i \aa_i^T
\right),
\end{equation}
where $\Delta_N:=\{\vec{n}\in\mathbb{Z}_{\geq 0}^m:\ \sum_{i=1}^m n_i=N\}$ is the standard discrete $N-$simplex,
and $n_i$ represents the number of trials to perform at the $i$th design point.
While Problem~\eqref{theBasicProblem_discrete} is of immediate relevance for the experimenter,
this problem has a hard combinatorial structure;
in particular, 
it contains as a special case the problem of exact $\vec{c}-$optimality, which was proved to be NP-hard in~\cite{CH12}.
Therefore, it is almost impossible to certify global optimality of a design $\vec{n}$,
except for small instances, when a mixed integer second order cone programming solver can be used~\cite{SH15}.
To overcome this issue, the classical machinery of \emph{approximate design theory} proposes to
introduce a continuous variable $w_i=\frac{n_i}{N}$ and to relax the integer constraints ``$N w_i\in\mathbb{Z}$'',
which leads to the convex optimization problem~\eqref{theBasicProblem}.
In practice, the solution of Problem~\eqref{theBasicProblem} gives a lower bound on
the optimal value of~\eqref{theBasicProblem_discrete}. This can be used to
ascertain the quality of an exact design~$\vec{n}$, which
can typically be computed by using heuristic methods,
such as exchange algorithms (see, e.g.~\cite{AD92}) or, as recently proposed, with particle swarm optimization~\cite{LMW18}.
Alternatively, rounding methods
can be used to turn an approximate design~$\vec{w}^*$ (i.e., a solution to Problem~\eqref{theBasicProblem}) into a good exact design~$\vec{n}\in\Delta_N$,
which works particularly well when the total number $N$ of trials is large~\cite{PR92}.
For more details on the subject, we refer the reader to
the monographs of Fedorov~\cite{Fed72}
or Pukelsheim~\cite{Puk93}.

\bigskip
Many different approaches have been proposed to solve Problem~\eqref{theBasicProblem}.
The traditional methods are the Fedorov-Wynn type vertex-direction
algorithms~\cite{Fed72,Wyn70}
and the closely related vertex exchange methods~\cite{Boh86},
the multiplicative weight update algorithms~\cite{STT78,Yu10a},
and interior point methods based on semidefinite programming~\cite{FL00} or second-order cone programming~\cite{Sag11} formulations.
Recent progress in this area has been obtained by employing hybrid methods that alternate between steps of the aforementioned algorithms
(the cocktail algorithm~\cite{Yu11}), or by using randomization~\cite{HFR18}.

\paragraph{Contribution and Organization.}

The main contribution of this article is a new reformulation of Problem~\eqref{theBasicProblem}
as a convex, unconstrained optimization problem, which
brings to light a strong connection with the well-studied problem of \emph{group lasso regression}~\cite{yuan2006model}.
The particular structure of the new formulation also suggests algorithmic ideas based on proximal decomposition methods,
which already proved to be very useful in machine-learning and signal processing applications~\cite{beck2009fast,combettes2011proximal,bach2012optimization}.
An appealing property of these methods is that they come with rigorous convergence guaranties, and yield sparse iterates very quickly, corresponding
to designs with only a few support points.

The rest of this paper is organized as follows. In Section~\ref{sec:ProblemReformulation} we give more background on 
Problem~\eqref{theBasicProblem}, and show how this problem can be reformulated as an unconstrained convex optimization
problem involving a \emph{squared group lasso penalty}. Then,
we characterize the proximity operator of
this penalty in Section~\ref{sec:convex}. This makes it possible to
use a new class of algorithms, described in Section~\ref{sec:algos}, to solve the reformulated problem.
Finally, Section~\ref{sec:experiments}
presents some numerical experiments comparing performances of the proposed algorithm to existing approaches.


\section{Problem Reformulation} \label{sec:ProblemReformulation}

\subsection{The Bayes $A_K-$optimal design problem}
For the sake of completeness, we first explain the derivation of Problem~\eqref{theBasicProblem_discrete} and its relaxation
for approximate designs, Problem~\eqref{theBasicProblem}.
The experimental design is specified by a vector $\vec{n}=(n_1,\ldots,n_m)\in\mathbb{Z}_{\geq 0}^m$,
which indicates the number of replications at the $i$th design point.
Specifically, we obtain random observations
$$
y_{ij} = \aa_i^T\vec{\theta} + \epsilon_{ij},\quad \forall i\in[m], \forall j\in [n_i],
$$
where the measurements are unbiased (i.e., $\E[\epsilon_{ij}]=0$),
uncorrelated (i.e., $(i,j)\neq (k,\ell) \implies \E[\epsilon_{ij} \epsilon_{k\ell}]=0$),
and the variance is known: $\E[\epsilon_{ij}^2]=\sigma^2$.
We further assume that the experimental resources are limited by a \emph{budget} on the total number $N$ of trials, 
that is, $\sum_{i=1}^m n_i=N$ must hold.

\medskip
Denote by $\yy$ the vector of $\R^m$ with the averaged
observations at each location, that is, $y_i = \frac{1}{n_i}\sum_{j=1}^{n_i} y_{ij}$
(and $y_i$ can be set to some arbitrary constant whenever $n_i=0$).
Then, in vector notation, we have
$$\yy = A \tth + \eeps,$$
where $A=[\aa_1,\ldots,\aa_m]^T\in\R^{m\times n}$,
and the averaged random vector $\vec{\epsilon}$ 
with elements $\epsilon_i = \frac{1}{n_i}\sum_{j=1}^{n_i} \epsilon_{ij}$
satisfies $\E[\vec{\epsilon}]=\vec{0}$, $\V[\vec{\epsilon}]=\E[\vec{\epsilon}\vec{\epsilon}^T] = \sigma^2\Diag(\vec{n})^{-1}.
$
(We adopt the convention $\frac{1}{0}=+\infty$, so $n_i=0$ implies that $\epsilon_i$ has an
infinite variance, which is consistent with the fact that $y_i$ is basically unobserved.) Further, we recall
that we have a prior observation $\tth_0=\tth+\vec{\eta}$, for some random vector $\vec{\eta}\in\R^n$ satisfying $\E[\vec{\eta}]=\vec{0},\ \E[\vec{\eta}\vec{\eta}^T]=\Sigma$, and $\E[\vec{\eta}\vec{\epsilon}^T]=0$. 

\bigskip
We know from the Gauss Markov theorem (see, e.g.~\cite{Puk93}) that
the best linear unbiased estimator (BLUE) for $\tth$
solves the least squares problem
$$
\underset{\tth \in \R^n}{\vec{\operatorname{minimize}}}\quad \frac{1}{\sigma^2} (A\tth-\yy)^T \Diag(\nn) (A\tth-\yy) + (\tth-\tth_0)^T\Sigma^{-1}(\tth-\tth_0),
$$
which admits the closed-form solution
$
\hat{\tth}:=M(\nn)^{-1} (A^T \Diag(\nn) \frac{\yy}{\sigma^2} + \Sigma^{-1} \tth_0),
$
where
$
M(\nn):= \frac{1}{\sigma^2} A^T \Diag(\nn) A + \Sigma^{-1} 
$
is the information matrix of the design. For the remaining of this paper, we focus on the \emph{approximate design} problem. As explained in the introduction,
this simply means that we introduce the continuous variable $\vec{w}=\frac{1}{N}\vec{n}\in\Delta$, and that we ignore
the restriction that $N w_i$ should be integer. So we define the (approximate) information matrix, by $M_N(\ww):= M(N\, \vec{w})$, $\forall \vec{w}\in\Delta$.
For the sake of notation, we find convenient to introduce the symbol $\sigma_N^2 = \frac{\sigma^2}{N}$, so 
$$M_N(\vec{w})= \frac{1}{\sigma_N^2} A^T \Diag(\ww) A + \Sigma^{-1}= \frac{1}{\sigma_N^2} \sum_{i=1}^m w_i \aa_i \aa_i^T + \Sigma^{-1} .$$
The Bayes $A-$optimal design problem is to select find $\ww$ minimizing $\Phi_A(M_N(\ww))$, where the criterion of $A-$optimality is $\Phi_A: M\mapsto \trace M^{-1}.$
Geometrically, this corresponds to minimizing the diagonal of the bounding box of 
confidence ellipsoids for $\hat{\vec{\theta}}$ (provided $\vec{\epsilon}$ is normally distributed), cf.~\cite{Sag11}.
Also, note that we recover the standard (non-Bayesian) A-optimal design problem when no prior is available, i.e., $\Sigma^{-1}\to 0$.


\medskip
More generally, the criterion of $A_K$-optimality is defined by
$$
\Phi_{A_K}: \mathbb{S}_{++}^n \to \R,\quad  M\mapsto \trace K^T M^{-1} K
$$
for some matrix $K\in\R^{m\times r}$. Clearly, the standard criterion of $A-$optimality is a particular case of $\Phi_{A_K}$, obtained by setting $K$ to the identity matrix.
Now, assume the experimenter wants to estimate a vector $\zzh=K^T\vec{\theta}$ for some matrix $K\in\R^{n\times r}$. Then, the BLUE for $\zzh$ is $\hat{\zzh}=K^T\hat{\tth}$
 and has variance-covariance matrix $\V[\hat{\zzh}]=K^T M_N(\vec{w})^{-1} K$.
Hence, we have $\Phi_{A_K}(M_N(\vec{w}))=\sum_{i=1}^r \V[\hat{\zeta_i}]$,
which shows that a Bayes $A_K$-optimal design (i.e., a design $\vec{w}^*$ solving Problem~\eqref{theBasicProblem})
minimizes the sum of the variances of the BLUE estimator.

We conclude this part by mentioning another common situation that leads to a Problem of the form~\eqref{theBasicProblem_discrete}.
Assume the experimenter wants to \emph{predict} the quantities $\eta(x):=\phi(\xx)^T\tth$, $\forall \xx \in \mathcal{X}$,
 where $\mathcal{X}$ is a compact set and $\phi:\mathcal{X}\to\R^n$ is a continuous map.
 Then, it is well known that $\hat{\eta}(\xx):=\phi(\xx)^T \hat{\vec{\theta}}$ is the best linear unbiased predictor (BLUP) for $\eta(\xx)$, and its variance is 
 $\phi(\xx)^T M_N(\vec{w})^{-1} \phi(\xx)$. If $\mu$ is a measure over $\mathcal{X}$ weighing the interest of the experimenter
 to predict $\eta$ at $\xx\in\mathcal{X}$, a natural criterion to consider is the \emph{integrated mean squared error},
 $\operatorname{IMSE}(\ww):= \int_{\xx\in\mathcal{X}} \phi(\xx)^T M_N(\ww)^{-1} \phi(\xx)\, d\mu(\xx)$\footnote{When $\mu$ is the uniform measure over the design space,
 we point out that the IMSE criterion is sometimes called $I$-optimality, or $IV$-optimality (for integrated variance).}.
 The minimization of $\operatorname{IMSE}(\ww)$
 can be cast as an $A_K$-optimal design problem, because:
 \begin{align*}
 \operatorname{IMSE}(\ww) &= \int_{\xx\in\mathcal{X}} \phi(\xx)^T M_N(\ww)^{-1} \phi(\xx)\, d\mu(\xx)\\
                          &= \trace M_n(\ww)^{-1} KK^T = \Phi_{A_K}(M_N(\ww)),\vspace{-1em}
 \end{align*}
where $KK^T$ is a Cholesky decomposition of the positive symmetric definite matrix $\int_{\xx\in\mathcal{X}} \phi(\xx) \phi(\xx)^T\, d\mu(\xx)$.
We point out that large scale problems involving the minimization of $\operatorname{IMSE}(\ww)$ recently arose for the sequential design of computer experiments
with Gaussian process as a metamodel, when a truncated Karhunen-Lo\`eve expansion is used to approximate the covariance kernel; see~\cite{gauthier2016optimal,gauthier2017convex,SagnolHegeWeiser2016}.

\subsection{Reformulation as an unconstrained convex problem}

Now, take a linear estimator $\hat{\vec{\zeta}}=X\yy+H\tth_0$ of $\zzh=K^T\vec{\theta}$ for some matrices $X \in \R^{r\times m}$ and $H \in \R^{r \times n}$.
This estimator is unbiased if and only if $XA+H=K^T$, 
and we have $\V[\hat{\zzh}]= \sigma_N^2 X \Diag(\ww)^{-1} X^T + H\Sigma H^T$. 
By the Gauss Markov theorem, minimizing over $X$ and $H$ the quantity $\sum_{i=1}^n \V[\hat{\zeta}_i] = \trace \V[\hat{\zzh}]$ such that $\hat{\vec{\zeta}}$ is unbiased, leads to the BLUE estimator,
in which case we have already seen that $\Phi_{A_K}(M_N(\vec{w}))=\sum_{i=1}^r \V[\hat{\zeta_i}]$. Hence, using the computed variance estimate, the Bayes $A_K$-optimal design is obtained by minimizing
further with respect to $\ww$, i.e., it can be obtained by solving
the following optimization problem
\begin{align}
\underset{\ww\in\R^m, X \in \R^{r\times m}, H \in \R^{r \times n}}{\vec{\operatorname{minimize}}} \ &\quad  \trace \sigma_N^2 X \Diag(\ww)^{-1} X^T + H\Sigma H^T \label{AOPT-SOCP}\\
s.t. \ & \quad XA+H=K^T\nonumber, \quad \ww\geq \vec{0},\quad  \sum_{i=1}^m w_i=1. \nonumber
\end{align}

The objective function is convex, as it can be written as
$\|H \Sigma^{1/2}\|_F^2 + \sigma_N^2 \sum_i \frac{\|\vec{x}_i\|^2}{w_i}$,
where $\vec{x}_i$ is the $i$th column of $X\in\R^{r\times m}$,
which is the sum of a convex quadratic and the perspective functions $(\vec{x}_i,w_i)\mapsto\frac{\|\vec{x}_i\|^2}{w_i}$
of  $\vec{x}_i\mapsto \|\vec{x}_i\|$; see~\cite{BV04}.
We also point out that this problem can be reformulated as a second order cone program (SOCP); see~\cite{Sag11}.

For a fixed $X$,
consider the function $J\colon \vec{w}\mapsto \sum_{i=1}^m \frac{\|\vec{x}_i\|^2}{w_i}$ from $\R_+^m$ to $\R \cup +\infty$. We use the convention that $\|\vec{x}_i\| / 0 = 0$ whenever
$\|\vec{x}_i\| = 0$ for $i = 1,\ldots, m$ which amounts to sum over indices with nonzero numerators: $J\colon \vec{w}\mapsto \sum_{i: \|\vec{x}_i\| > 0} \frac{\|\vec{x}_i\|^2}{w_i}$ and ensures that $J$ is well defined.
We also assume that $\sum_i \|\vec{x}_i\| > 0$ so that $J$ is not constant.
In this case, $J$ is minimized
over the probability simplex for
$w_i^* = \frac{\|\vec{x}_i\|}{\sum_i \|\vec{x}_i\|}$, $i=1,\ldots,m$.
In other words
\begin{align}
				\label{eq:minSimplex}
				\vec{w}^* \in \arg\min \quad \left\{J(\vec{w}),\quad \mathrm{s.t.} \quad w_i \geq 0,\, i =1 ,\ldots,m,\quad\sum_{i=1}^m w_i = 1  \right\}.
\end{align}
Since $J$ is convex, this can be verified by checking the first order Karush-Kuhn-Tucker (KKT) conditions:
note that $\vec{w}^*$ is feasible for (\ref{eq:minSimplex}) and $F$ is differentiable at $\vec{w}^*$ and for all $i = 1, \ldots, m$,

\begin{align}
				\label{eq:KKTSimplex}
				\left. \frac{\partial J(\vec{w})}{\partial w_i}\right|_{\vec{w}=\vec{w}^*} = 
				\begin{cases}
								-\frac{\|\vec{x}_i\|^2}{w_i^{*2}} =- (\sum_i \|\vec{x}_i\|)^2&\text{ if } w_i^* > 0;\\
								0& \text{ otherwise.}
				\end{cases}
\end{align}
Equation (\ref{eq:KKTSimplex}) is precisely KKT optimality condition at $\vec{w}^*$ for Problem~\eqref{eq:minSimplex} (see \textit{e.g.} \cite[Example 3.4.1]{bertsekas1999nonlinear}).
Plugging the expression of $\vec{w}^*$ into~\eqref{AOPT-SOCP}, we obtain the following problem:
\begin{align}
\underset{X, H}{\vec{\operatorname{minimize}}} \ &\quad  \|H \Sigma^{1/2}\|_F^2 + \sigma_N^2 (\sum_i \|\vec{x}_i\|)^2\qquad
s.t. \ & \quad XA+H=K^T.\nonumber
\end{align}
We can eliminate the variable $H$ from this problem, which leads to an unconstrained, convex optimization problem with a nice structure.
We summarize our findings in the next proposition:
\begin{proposition}\label{prop:SGL}
 Consider the optimization problem
\begin{align}
\rho\quad:=\quad\underset{X}{\vec{\operatorname{min}}}  &\quad  \|(XA-K^T) \Sigma^{1/2}\|_F^2 + \sigma_N^2 (\sum_i \|\vec{x}_i\|)^2.
\label{eq:SGL}
\end{align} 
Then, $\rho$ is equal to the optimal value of Problem~\eqref{theBasicProblem}, and if
$X^*=[\xx_1^*,\ldots,\xx_m^*]$ solves Problem~\eqref{eq:SGL}, then the design defined by $w_i^* = \frac{\|\xx_i^*\|}{\sum_j \|\xx_j^*\|}$
is Bayes $A_K-$optimal.
\end{proposition}

%
If the square was removed from $(\sum_i \|\xx_i \|)^2$, this last term would be similar to a group lasso penalty~\cite{yuan2006model}. From a practical perspective, the main interest of this reformulation is that it paves the way toward the use of well established first order methods to tackle such problems \cite{beck2009fast,combettes2011proximal,bach2012optimization}.

\bigskip
Interestingly, the idea of using a group lasso to design experiments has already been proposed in~\cite{TM13}. However, this paper 
justified the group lasso approach heuristically, in order to select the support points of an exact design.
Indeed, group lasso regression was designed to recover an approximate solution of an equation 
of the form $\sum_i A_i' \xx_i \simeq \yy'$ with only a small number of nonzero blocks $\xx_i$.
It is widely known that optimal designs often have a small number of support points,
and hence correspond to an estimator $\hat{\vec{\theta}}=X\vec{y}$ with many columns of $X$ equal to $\vec{0}$. Therefore,
group lasso regression can be used to find sparse estimators that satisfy
approximately the unbiasedness property: $XA\simeq K^T$.
The result of Proposition~\eqref{prop:SGL}
shows that in fact, one obtains
an exact reformulation of the Bayes $A_K$-optimal design problem by squaring the penalty.

\section{Convex analysis of the squared group lasso penalty} \label{sec:convex}

Throughout the rest of this article, we set for all $X \in \R^{r \times m}$ where, for each $i=1,\ldots, m$,  $\vec{x}_i \in\R^n$ is the $i$th column of $X$:
\begin{itemize}
				\item $f \colon X \mapsto \|(XA-K^T)\Sigma^{1/2}\|_F^2$.
				\item The norm $\Omega \colon X \mapsto \sum_{i=1}^m \|\vec{x}_i\|$ and its dual norm, $\Omega^* \colon X \mapsto \max_{i=1\ldots m} \|\vec{x}_i\|$.
				\item $g \colon \R^{r\times m} \mapsto \R$ with $g(X) = \frac{1}{2} \Omega(X)^2$.
\end{itemize}
We use the usual Euclidean scalar product on matrices. With these notations, problem (\ref{eq:SGL}) may be rewritten as
 				\begin{align}
 								\mathrm{min}_X \quad& F(X):=\ f(X) + 2\, \sigma_N^2\, g(X)
 								\label{eq:SGLbis}
 				\end{align}
Note that the function $g$ is convex and that the outer square destroys the separability of the inner sum in $g$, unlike standard group lasso penalty. 
This leads to non trivial optimization developments.
%
%
%
The reader is referred to \cite{rockafellar1970convex,borwein2010convex} for detailed exposition of convex analysis related material.
\begin{lemma}[Subgradient and conjugate]
				Let $g \colon \R^{r\times m} \mapsto \R$ be such that $g(X) = \frac{1}{2} \left( \sum_{i=1}^m \|\vec{x}_i\| \right)^2$ where $\vec{x}_i$ is the $i$th column of $X$. Then we have the following formula for the subgradient and the Legendre transform of $g$ denoted by $g^*$:
				\begin{align*}
								&\forall X \in \R^{r\times m}, \, \partial g (X) = \left( \sum_{i=1}^m \|\vec{x}_i\| \right)\left[ \vec{v}_1 \vec{v}_2 \ldots \vec{v}_m \right],\, \vec{v}_i \in \partial \|\vec{x}_i\|,\, i=1,\ldots,m.\\
								&\forall Z \in \R^{r\times m},\, g^*(Z) = \frac{1}{2} \max_{i=1,\ldots,m}\left\{\|\vec{z}_i\|^2  \right\}.
				\end{align*}
				\label{lem:subgradientG}
\end{lemma}
\begin{proof}
				We mostly follow \cite{bach2012optimization} and provide detailed arguments. We set $\Omega \colon X \mapsto \sum_{i=1}^m \|\vec{x}_i\|$ which is a norm. Its dual norm is $\Omega^* \colon Z \mapsto \max_{i=1\ldots m} \|\vec{z}_i\|$. Fix any $Z \in \R^{r\times m}$, we have for any $X \in \R^{r\times m}$,
				\begin{align*}
								\left\langle X,Z \right\rangle - g(X) &\leq \Omega^*(Z)\Omega(X) - \frac{1}{2} \Omega(X)^2 \leq \frac{1}{2} \Omega^*(Z)^2.
				\end{align*}
				Setting $X = \Omega^*(Z) \partial \Omega^*(Z)$, we obtain $\Omega(X) = \Omega^*(Z)$ and $\left\langle Z, X\right\rangle = \Omega^*(Z)^2$ so that the above holds with equality. This entails that $g^* = \frac{1}{2} (\Omega^*)^2$ which is precisely the claimed formula for the conjugate function. Now symmetrically, for any fixed $X \in \R^{r\times m}$, setting $Z = \Omega(X) \partial \Omega(X)$ we obtain $\Omega^*(Z) = \Omega(X)$ and $\left\langle Z, X\right\rangle = \Omega(X)^2 =  \frac{1}{2}\left( \Omega(X)^2 + \Omega^*(Z)^2 \right)$ which shows by \cite[Theorem 23.5]{rockafellar1970convex} that $Z \in \partial g(X)$. The claimed form of the subgradient follows because $\Omega$ has a structure of separable sum, see \cite[Theorem 23.8]{rockafellar1970convex}.
\end{proof}
Given $t> 0$, the following lemma describes how to compute the proximity operator of $X\mapsto t\,g(X)$:
$$
\operatorname{prox}_{tg}(V):=\arg\!\min_{X}\, t\cdot g(X) + \frac{1}{2} \|X-V\|_F^2.
$$

%

\begin{algorithm}[t]
\caption{($\operatorname{prox}$-operator)}
\label{alg:prox}
\begin{algorithmic}[1]                    
    \alginput $t,V$
    \algoutput $\operatorname{prox}_{tg}(V)$
    \medskip
    \myState{Order the columns of $V$ by decreasing order of norm, such that $\|\vv_1\|  \geq \cdots \geq \|\vv_m\|$ (store the corresponding permutation)}
    \myState{Make a binary search to find the largest $k\leq m$ such that
$\|\vv_k\| \geq \frac{t}{tk+1} \sum_{i=1}^k \|\vv_i\|.$}
    \myState{Set $\xx_i = \left( 1-\frac{t}{tk+1} \sum_{j=1}^k \frac{\|\vv_j\|}{\|\vv_i\|} \right) \vec{v}_i$, for $i=1,\ldots,k$.} 
    \myState{Set $X=[\xx_1,\ldots,\xx_k,\vec{0},\ldots,\vec{0}]\in\R^{r\times m}$, and permute the columns according to the inverse permutation obtained from the first step.}
    \myState{\Return $X$}
\end{algorithmic}
\end{algorithm}
\begin{lemma}[Proximity operator]
				Let $V \in \R^{r \times m}$ and $\vec{v}_i \in \R^n$ be its columns for $i=1,\ldots,m$ and $t > 0$. Then Algorithm \ref{alg:prox} computes $\operatorname{prox}_{tg}(V)$.
				\label{lem:proxG}
\end{lemma}
\begin{proof}
				First note that $k$ is well defined since the condition obviously holds for $k = 1$. Furthermore, for all $i \leq k$ we have $\|\vec{v}_i\| \geq \|\vec{v}_k\| \geq \frac{t}{tk+1} \sum_{j=1}^k \|\vv_j\|$. Note also that the proposed definition for $\vec{x}_i$ ensures that $i > k$ for all $i$ such that $\vec{v}_i = 0$ so that there is no division by $0$ and $\vec{x}_i = 0$ whenever $\vec{v}_i = 0$ . We just need to check that $(V-X)/t \in \partial g(X)$. We have
				\begin{align*}
								\sum_{i=1}^m \|\vec{x}_i\| = \sum_{i =1}^k \|\vec{v}_i\| - \frac{kt}{tk + 1} \sum_{j = 1}^k \|\vec{v}_j\| = \frac{1}{tk + 1} \sum_{j=1}^k \|\vec{v}_j\|.
				\end{align*}
				We now consider several cases.
				\begin{itemize}
								\item If $i > k$ and  $\vec{v}_i = 0$, then $\vec{x}_i = 0$ and $(\vec{v}_i - \vec{x}_i ) / t = 0 \in \partial \|\vec{x}_i\|$.
								\item If $i > k$ and  $\vec{v}_i \neq 0$, then $\vec{x}_i = 0$ and it holds that $\frac{\vec{v}_i}{\|\vec{v}_i\|} \in \partial \|\vec{x}_i\|$ and $\frac{\vec{v}_i - \vec{x}_i}{t} = \frac{ \vec{v}_i }{ \|\vec{v}_i\|} \left( \sum_{i=1}^m  \|\vec{x}_i\| \right)$.
								\item If $i \leq k$, then $\left( 1-\frac{t}{tk+1} \sum_{j = 1}^k \frac{\|\vv_j\|}{\|\vv_i\|} \right) \geq 0$ and $\vec{v}_i \neq 0$ so that $\frac{\vec{v}_i}{\|\vec{v}_i\|} \in \partial \|\vec{x}_i\|$. We also have $\frac{\vec{v}_i - \vec{x}_i}{t} = \frac{ \vec{v}_i }{ \|\vec{v}_i\|} \left( \sum_{i=1}^m  \|\vec{x}_i\| \right)$.
				\end{itemize}
				 This shows that the proposed $X$ satisfies the subdifferential characterization in Lemma \ref{lem:subgradientG} and the result follows.
\end{proof}

\section{Algorithms} \label{sec:algos}

\subsection{Proximal decomposition methods} \label{sec:prox_dec}
In this section we describe convex optimization algorithms dedicated to structured ``smooth plus nonsmooth'' problems with easily computable proximity operator. Further details and historical comments are found in \cite{combettes2011proximal,beck2009fast,bach2012optimization}. 
On the one hand, we have $\nabla f(X) = 2 (XA-K^T)  \Sigma A^T.$
On the other hand, Lemma~\ref{lem:proxG} ensures that $\operatorname{prox}_{tg}(V)$, can be computed by Algorithm~\ref{alg:prox}. These are the building blocks of proximal decomposition algorithms. We describe the backtracking line search variants of the Forward-Backward algorithm and 
FISTA algorithm.
Backtracking line search ensures minimal parameter tuning beyond the initialization. One can use a fixed step size $1/L$ instead, where $L=\mathrm{trace}(A\Sigma A^T)$ is the Lipschitz constant of $\nabla f$. 

\paragraph{Forward-Backward algorithm}: This is the simplest proximal decomposition algorithm. More details can be found in \cite{combettes2011proximal,beck2009fast}.
\begin{algorithm}[t]
\caption{Forward-Backward with Backtracking line search}
\label{alg:FB}
\begin{algorithmic}[1]                    
    \alginput $X_0 \in \R^{r \times m}$, $\eta > 1$, $L_0 > 0$
    \medskip
		\For{$k=1,2\ldots$}{
			\State Find the smallest $i \in \N$ such that, 
			with $\bar{L} = \eta^{i} L_{k-1}$ and $P= \mathrm{prox}_{2 \sigma_N^2g/\bar{L}}\left( X_{k-1} - \frac{1}{\raisebox{-1mm}{{\small $\bar{L}$}}} \nabla f(X_{k-1}) \right)$,
			$$F(P)\leq\quad f(X_{k-1}) + \left\langle\nabla f(X_{k-1}),P- X_{k-1}\right\rangle + 2 \sigma_N^2 g(P) + \frac{\bar{L}}{2}\|P - X_{k-1}\|_F^2$$
		\State Set $L_k = \bar{L}$ and $X_k = P$\;
		\EndFor}
\end{algorithmic}
\end{algorithm}
Known properties for this algorithm include the following:
\begin{itemize}
				\item The sequence $\left( X_k \right)_{k \in \N}$ converges to a solution of problem (\ref{eq:SGLbis}) and for any $X^*$ solution of the problem, the sequence $\left(\|X_k - X^*\|  \right)_{k \in \N}$ is non increasing.
				\item The objective function 
				$F(X_k)$ is monotonically decreasing along the sequence and we have for all $k\in\N$,
								\begin{align*}
												F(X_k) - \rho \leq \frac{\eta L \|X_0 - X^*\|_F^2}{2k},
								\end{align*}
								for any $X^*$ solution to Problem \ref{eq:SGLbis} (see \cite{beck2009fast}). Here $L=\mathrm{trace}(A\Sigma A^T)$ is the Lipschitz constant of the gradient of $f$ (with respect to Frobenius norm).
				\item For all $k \in \N$, we have $L_k \leq \max\left\{ L_0,\eta L \right\}$
\end{itemize}
\paragraph{FISTA acceleration:} It is known since the seminal work of Nesterov \cite{nesterov1983method} that $O(1/k)$ is not optimal for convex optimization with gradient methods. Accelerated methods exist with a faster $O(1/k^2)$ convergence rate. We now describe 
the FISTA algorithm \cite{beck2009fast} which belongs to this family of methods and is applicable to problem (\ref{eq:SGLbis}).
\begin{algorithm}[t]
\caption{Fista with Backtracking line search}
\label{alg:FISTA}
\begin{algorithmic}[1]                    
				\alginput $X_0 \in \R^{r \times m}$, $\eta > 1$, $L_0 > 0$, $Y_1 = X_0$, $t_1 = 1$.
    \medskip
		\For{$k=1,2\ldots$}{
			\State Find the smallest $i \in \N$ such that, with $\bar{L} = \eta^{i} L_{k-1}$ and $P= \mathrm{prox}_{2 \sigma_N^2g/\bar{L}}\left( Y_{k} - \frac{1}{\raisebox{-1mm}{{\small $\bar{L}$}}} \nabla f(Y_{k}) \right)$,
			$$F(P)\leq\quad f(Y_{k}) + \left\langle\nabla f(Y_{k}),P- Y_{k}\right\rangle + 2 \sigma_N^2 g(P) + \frac{\bar{L}}{2}\|P - Y_{k}\|_F^2$$
		\State Set $L_k = \bar{L}$ and $X_k = P$ and $t_{k+1} = \frac{1 + \sqrt{1 + 4 t_k^2}}{2}$ and
		$$Y_{k+1} = X_k + \left( \frac{t_k - 1}{t_{k+1}} \right)\left( X_k - X_{k-1} \right).$$
		\EndFor}
\end{algorithmic}
\end{algorithm}
Contrary to Forward Backward algorithm, FISTA algorithm does not provide a monotonically decreasing sequence of objective values, and convergence of the sequence $(X_k)_k$ is not known yet for this precise version, although it is for very close variants \cite{chambolle2015convergence}. The main feature of FISTA is the following complexity estimate, for any $k \in \N$, 
\begin{align*}
				F(X_k) - \rho \leq \frac{2\eta L \|X_0 - X^*\|_F^2}{(k+1)^2},
\end{align*}
for any $X^*$ solution to Problem \ref{eq:SGLbis} (see \cite{beck2009fast}). Here $L=\mathrm{trace}(A\Sigma A^T)$ is the Lipschitz constant of the gradient of $f$ (with respect to Froebenius norm).

The complexity of one iteration of either Forward-Backward or FISTA algorithm is dominated by the cost of computing $\nabla f(X) = 2(XA - K^T) \Sigma A^T$ which can be done in $O(r \times n \times m)$ operations (this is the cost of multiplication of $X \in \R^{r\times m}$ and $A \in \R^{m\times n}$
and multiplying the result by $\Sigma A^T \in \R^{n \times m}$). For a typical situation with $r = n$, this is $O(n^2 m)$. The cost of computing the proximity operator is negligible.

\subsection{Block coordinate descent} \label{sec:abcd}

An alternative to solve the unconstrained optimization problem~\eqref{eq:SGL}  is to iteratively solve the problem for one particular block $\xx_i$, while keeping all other blocks fixed. This idea is attractive, because optimization over a single block admits a simple closed-form solution, as the following proposition shows.

\begin{proposition}
 Let $i\in[m]$, and let the $\xx_j$'s be a fixed vectors in $\R^n$ ($\forall j\in[m], j\neq i$). We consider 
 the variant of Problem~\eqref{eq:SGL} in which we minimize the criterion with respect to the block of variables $\xx_i$ only, that is:
\begin{align}\label{SGL}
				\mathrm{min}_{\xx_i} \ &\quad h_i(\xx_i):= \|( \xx_i \aa_i^T + R) \Sigma^{1/2}\|_F^2 + \sigma_N^2 (\|\vec{x}_i\| + \beta)^2,
\end{align}
 where $R:=\sum_{j\neq i} \xx_j \aa_j^T-K^T$
 and $\beta = \sum_{j\neq i} \|\xx_j\|$.
 The optimal solution of this problem is given by $\xx_i^*=\vec{0}$ whenever $R \Sigma \aa_i=\vec{0}$ and otherwise,
 $$
 \xx_i^* = -\frac{1}{\aa_i^T  \Sigma \aa_i + \sigma_N^2} \cdot \max\left\{1 -\sigma_N^2 \frac{\beta}{\|R\Sigma \aa_i\|}, 0   \right\} \cdot  R \Sigma \aa_i,
$$
\label{prop:blockMin}
\end{proposition}


\begin{proof}
 We can rewrite the function to minimize as
$$ h_i(\xx_i) = 
  \aa_i^T  \Sigma \aa_i \|\xx_i\|^2 + 2 \xx_i^T R \Sigma \aa_i + \|R \Sigma^{1/2} \|_F^2 +  \sigma_N^2 (\|\vec{x}_i\| + \beta)^2.$$
	Expanding the square, the subgradient sum rule \cite[Theorem 23.8]{rockafellar1970convex} gives the following expression for the subgradient $\partial (\|\vec{x}\| + \beta)^2 = 2\partial \|\vec{x}\|(\|\vec{x}\| + \beta)$,
	hence the subgradient of $h_i$ has the following form:
$$\partial h_i(\xx_i) = \left\{\begin{array}{ll} 
         \quad \left\{2 \big[ (\aa_i^T  \Sigma \aa_i) \xx_i + R\Sigma\aa_i + \sigma_N^2 (\|\xx_i\| + \beta) \frac{\xx_i}{\|\xx_i\|}  \big] \right\} &\ \text{if }\xx_i\neq\vec{0}; \\[1em]
         \quad \left\{2 \big[  R\Sigma\aa_i + \sigma_N^2 \vec{u}\big]:\   \|\uu\| \leq \beta \right\} & \ \text{otherwise.}
         \end{array}\right.
$$

It remains to show that $\vec{0}\in\partial{h_i}(\xx_i^*)$. If $R\Sigma \vec{a}_i = \vec{0}$ then $\vec{x}_i = 0$ and the statement holds. Assume that $R\Sigma \vec{a}_i\neq\vec{0}$, we distinguish two cases.
\begin{itemize}
 \item If $(1 -\sigma_N^2 \frac{\beta}{\|R\Sigma \aa_i\|})>0$, then
 $$\xx_i^* = -\frac{1}{\aa_i^T  \Sigma \aa_i + \sigma_N^2} \cdot (1 -\sigma_N^2 \frac{\beta}{\|R\Sigma \aa_i\|}) \cdot  R \Sigma \aa_i\neq \vec{0} .$$
Substituting in the expression of $\partial h_i$, easy (though lengthy) calculations shows  that $\partial h(\xx_i^*)=\{\vec{0}\}$.
 
 \item Otherwise, we have $\|R\Sigma \aa_i\| \leq \sigma_N^2 \beta $ and $\xx_i^*=\vec{0}$.
 To see that $\vec{0}\in\partial{h_i}(\xx_i^*)$, we need
 a vector $\vec{u}$ such that $R\Sigma\aa_i + \sigma_N^2 \vec{u}=\vec{0}$ and $\|\uu\|\leq \beta$.
 This works for  $\uu = -\frac{1}{\sigma_N^2} R\Sigma\aa_i$.
\end{itemize}
\end{proof}

\paragraph{Alternating minimization:} Block coordinate methods are wide spread for large scale problems, see for example \cite{wright2015coordinate} for a recent overview. The idea is to update only a subset of variable at each iteration. The choice of the subset could be performed in various ways: at random with replacement, in a cyclic order, using random permutations.
We describe the block minimization algorithm which is well suited for our problem thanks to Proposition \ref{prop:blockMin}.
\begin{algorithm}[t]
\caption{Alternating Block Coordinate Descent}
\label{alg:ABCD}
\begin{algorithmic}[1]                    
		\alginput $X_0 \in \R^{r \times m}$, and denote by $\vec{x}_{i0}$, $i=1,\ldots, m$, its columns 
    \medskip
		\For{$k=1,2\ldots$}{
			\State Choose an integer $i \in \left\{ 1,\ldots,m \right\}$ (see the main text for different possibilities). 
			\State Set $\vec{x}_{ik} = -\frac{1}{\aa_i^T  \Sigma \aa_i + \sigma_N^2} \cdot \max\left\{1 -\sigma_N^2 \frac{\beta}{\|R\Sigma \aa_i\|}, 0   \right\} \cdot  R \Sigma \aa_i$\;
			where $R:=\sum_{j\neq i} \xx_{j(k-1)} \aa_j^T-K^T$ and $\beta = \sum_{j\neq i} \|\xx_{j(k-1)}\|$.
			\State Set $\vec{x}_{jk} = \vec{x}_{j(k-1)}$ for all $j \in \left\{ 1,\ldots, m \right\}$, $j \neq i$. 
		\EndFor}
\end{algorithmic}
\end{algorithm}
Implementing the alternating minimization algorithm requires to keep track of $XA\Sigma \in \R^{r \times n}$ (similarly as for computing $\nabla f$). Keeping track of this quantity when changing a single column can be done in $O(nr)$ operations. A full path through the $m$ columns can be done in $O(nrm)$ operations which is the same as for gradient based methods.

To our knowledge application of this algorithm to a problem of the form of (\ref{eq:SGL}) is new in the optimization literature. Indeed, alternating minimization and more generally block coordinate methods are not convergent in general, their use is limited to smooth problems or problems
with a separable sum structure. This is not the case because of the square in the last term of (\ref{eq:SGL}). 

To understand why block coordinate methods do not converge to global minima in general, consider the function $\phi \colon (x,y) \mapsto \max\left\{x + 2 y, -2x - y  \right\}$
Taking for any $t > 0$, $x = t$ and $y = -t$ and letting $t \to \infty$ shows that $\inf_{\R^2}\phi = -\infty$. Yet it can be checked that $0 = \arg\min_x \phi(x,0) = \arg\min_y \phi(0,y)$ so that the origin is actually a stationary point for the alternating minimization algorithm applied to $\phi$.

However problem (\ref{eq:SGL}) has an additional structure: the subgradient of its objective is a simple Cartesian product. Furthermore, partial minimization is strongly convex. Combining these properties leads to the following result which to our knowledge is new. This guaranty is weak, indeed, convergence of alternating minimization methods is a difficult matter for which only few results are known and virtually none outside of separable nonsmoothness.

\begin{proposition}
				The alternating minimization algorithm applied to problem (\ref{eq:SGL}) with blocks taken in a cyclic order or using random permutations, produces a decreasing sequence of objective function value and satisfies $F(X_k) \to \rho\quad \text{ as }\ \ k \to \infty$.	
				\label{prop:convergenceAltMin}
\end{proposition}
\begin{proof}
				Monotonicity is obvious here, we denote by $\tilde{\rho}$ the limiting value of the objective function along the sequence. The Cartesian product structure of the subgradient of $g$ in Lemma \ref{lem:subgradientG} entails that the subgradient of the objective of (\ref{eq:SGL}) has the same Cartesian product structure. This implies that if all the columns of $X \in \R^{r \times m}$ are blockwise optimal for problem  (\ref{eq:SGL}), then $X$ itself is the global optimum. This is because block optimality ensures that $\vec{0}$ belongs to each partial subgradients in (\ref{SGL}) and the global subgradient of (\ref{eq:SGL}) is the Cartesian product of the partial subgradients \cite[Corollary 10.11]{rockafellar2009variational}.
				
				Now the partial minimization in (\ref{SGL}) is $2 \sigma_N^2$-strongly convex. Hence, for all $k \in \N$,
				\begin{align*}
								F(X_{k}) -  F(X_{k+1}) 	\geq \sigma_N^2 \|X_{k+1} - X_{k}\|^2.
				\end{align*}
				So $\left\{ \|X_{k+1} - X_{k}\|^2 \right\}_{k\in \N}$ is summable and as $k \to \infty$, we have $\|X_{k+1} - X_{k}\| \to 0$. By monotonicity $\left\{ X_k \right\}_{k \in \N}$ is a bounded sequence since the objective in (\ref{eq:SGL}) is coercive. Let $\bar{X}$ be any accumulation point of the sequence (there exists at least one). 

				For cyclic or random permutation selections, since all blocks are visited every $m$ iteration, using the notation of Proposition \ref{prop:blockMin}, by continuity of the objective function, one must have for all $i = 1,\ldots,m$ that the quantity
				\begin{align*}
								\vec{x}_{ik} - \arg\min_{\xx_i} \|( \xx_i \aa_i^T + R_k) \Sigma^{1/2}\|_F^2 + \sigma_N^2 (\|\vec{x}_i\| + \beta_k)^2 \quad \underset{k \to \infty}{\to}\quad \vec{0},
				\end{align*}
				where $R_k:=\sum_{j\neq i} \xx_{jk} \aa_j^T-K^T$ and $\beta_k = \sum_{j\neq i} \|\xx_{jk}\|$. By continuity $\bar{X}$ must be blockwise optimal for (\ref{eq:SGL}) and hence global optimal so that $\tilde{\rho} = \rho$.
\end{proof}

\section{Numerical experiments} \label{sec:experiments}

\subsection{Instances}
As was done in~\cite{HFR18},
we report numerical experiments on two kinds of instances to test the performance of proximal decomposition methods to solve Problem~\eqref{eq:SGL}.
On the one hand, we generate random instances by sampling the elements of $A\in \R^{m \times n}$ independently from a standard normal distribution.
On the other hand, we compute Bayes $A-$optimal designs for quadratic regression over $[-1,1]^d$:
$$y(x) = \theta_0 + \sum_{i=1}^d \theta_i x_i + \sum_{1\leq i\leq j\leq d} \theta_{ij} x_i x_j + \epsilon.$$
So in practice, to construct the matrix $A$ we first form a regular grid $X=\{\xx_1,\ldots,\xx_m\}\subseteq [-1,1]^d$,
and for each $k\in[m]$ the $k$th row of $A$ is set to $$\aa_k^T=[1,(x_{ki})_{k=1,\ldots,d},(x_{ki}\, x_{kj})_{1\leq i\leq j \leq d}]\in\R^{n},$$
where $n=1+d+d(d+1)/2$.
In addition, for all our experiments, we set $K=\Sigma=I_n$, and $\sigma_N^2=0.01$.

\subsection{Algorithms}

We present results for the two proximal decomposition methods with backtracking line search
presented in Section~\ref{sec:prox_dec}, which we denote by \FB\ (for \emph{Forward-Backward}) and
\FISTA.
We also used two variants of the alternating block coordinate descent algorithm of~\ref{sec:abcd},
where blocks are selected in a fixed cyclic order ($\ABCDcy$), or according
to a new random permutation that is drawn at random every $m$ steps ($\ABCDrp$).

We compare these methods to a Fedorov-Wynn type vertex-direction method (\VDM),
which is, in fact, an adaptation of the celebrated Frank-Wolfe algorithm
for constrained convex optimization~\cite{FW56}.
Several variants exist to compute the step sizes $\alpha$ of this algorithm, in particular,
optimal step length can be used, see~\cite{HFR18}. However, no simple formula exists for the optimal step lengths in the case of \emph{Bayes} A-optimality, so we next describe
a method with backtracking line search, which also allows a more straightforward comparison with \FB\ and \FISTA.

\vspace{.1in}

\begin{algorithm}[t]
\caption{Vertex Direction Method with Backtracking line search}
\label{alg:VDM}
\begin{algorithmic}[1]                    
		\alginput $\eta > 1$, $L_0 > 1$, $\ww_0\in\Delta$
    \medskip
		\For{$k=1,2\ldots$}{
			\State Compute the vector $\vec{d}$ with elements $d_i=\frac{1}{\sigma_N^2}\|K^T M_N(\vec{w}_{k-1})^{-1} \aa_i\|^2$.
			\State Select $i^* \in \arg\max_i \{d_i\}$.
			\State Find the smallest $j \in \N$ such that, with $\bar{L} = \eta^{j} L_{k-1}$,
						$\alpha = \frac{1}{\raisebox{-1mm}{{\small $\bar{L}$}}}$
						and $\bar{\vec{w}}= (1-\alpha) \vec{w}_{k-1}  + \alpha \vec{e}_{i^*}$,
									  $$\Phi_{A_K}(\bar{\vec{w}}) \leq \Phi_{A_K}(\vec{w}_{k-1}) - \vec{d}^T (\bar{\vec{w}}-\vec{w}_{k-1}) 
									  + \frac{\bar{L}}{2} \| \bar{\vec{w}}-\vec{w}_{k-1} \|^2.$$
			\State Set $L_k=\bar{L}$ and $\vec{w}_k=\bar{\vec{w}}$.
		\EndFor}
\end{algorithmic}
\end{algorithm}

We will also compare to the multiplicative algorithm~\cite{STT78} (\MUL), where at each iteration,
we set 
\begin{align} \label{eq:gradient}
(\vec{d}_k)_i = -\frac{\partial \Phi_{A_k}(\vec{w}_k)}{\partial w_i}=\frac{1}{\sigma_N^2}\|K^T M_N(\vec{w}_{k})^{-1} \aa_i\|^2, 
\end{align}
and we perform the update $\vec{w}_{k+1} = \frac{\vec{w}_k \odot \vec{d}_k}{\vec{w}_k^T \vec{d}_k}$; here, the symbol $\odot$ is used for the Hadamard (elementwise) product of two vectors.

For both \VDM\ and \MUL, the cost of one iteration is dominated by the cost of computing $\vec{d}=-\nabla \Phi$ which requires the inversion of $M_N(\vec{w})$ with computational cost $O(n^3)$ and multiplication by $A^T$ which cost is $O(n^2 \times m)$ and dominates the overall cost of this gradient computation. 

For all algorithms, we used the constants $\eta=2$ and $L_0=1$ for backtracking line searches.
The initial designs were set to $\ww_0=\frac{1}{m} \vec{1}_m$ for $\VDM$ and $\MUL$, and 
we used the initial matrix $X_0=0\in\R^{r \times m}$ for the other algorithms.

In our experiments, $K$ is taken to be the identity so that all the algorithms have iteration complexity of order $O(n^2 m)$ and thus comparing the evolution of the cost along iterations of each algorithm provides a good intuition about their comparative performances. Note that for $\ABCDcy$ and $\ABCDrp$, we consider that one iteration is complete after going through a full cycle so that all the entries of $X$ are updated.

\subsection{Results}

To monitor the speed of convergence of the algorithms, 
we can compute the design efficiencies
$$\operatorname{eff}_{A_K} (\vec{w}_k) := \frac{\rho}{\Phi_{A_K}(\vec{w}_k)},$$
where $\rho = \inf_{\vec{w^*} \in \Delta} \Phi_{A_K}(\vec{w^*})$ 
was computed by letting the multiplicative algorithm run for a very long time.
On the graphics, we plot the quantity $\log_{10}(1-\operatorname{eff}_{A_K} (\vec{w}_k))$,
so a value of $-k$ corresponds to an efficiency of $1-10^{-k}$.
We will also use the following optimality measure:
$$
 s_k = \max_{i=1,\ldots,m}\ (\vec{d}_k)_i - \vec{w}_k^T \vec{d}_k,
$$
where $\vec{d}_k$ is the gradient of $-\Phi_{A_K}$ at $\vec{w}_k$, see~\eqref{eq:gradient}.
It is folklore (see~\cite{Puk93}) 
that this expression gives a duality bound on the \emph{efficiency} of the
design $\vec{w}_k$: $\operatorname{eff}_{A_K} (\vec{w}_k)\geq 1-\varepsilon_k$, where $\varepsilon_k:= \frac{s_k}{\Phi_{A_K}(\vec{w}_k)+s_k}$.
Furthermore, for any sequence $(\ww_k)_{k\in\mathbb{N}}$ of designs converging to an optimal design $\vec{w}^*$, it is known that $s_k$ converges to $0$, so the
lower bound $1-\varepsilon_k$ on the design efficiency converges to $1$. The algorithms \FB, \FISTA\ and \texttt{ABCD} do not directly involve iterates
$\vec{w}_k$, but we can compute the above efficiency bound by setting $\vec{w}_k=\|\vec{x}_k\|/\Omega(X)$.

 \begin{figure}[t]
 \begin{tabular}{cccl}
   \hspace{-3mm} \scalebox{0.37}{\includegraphics{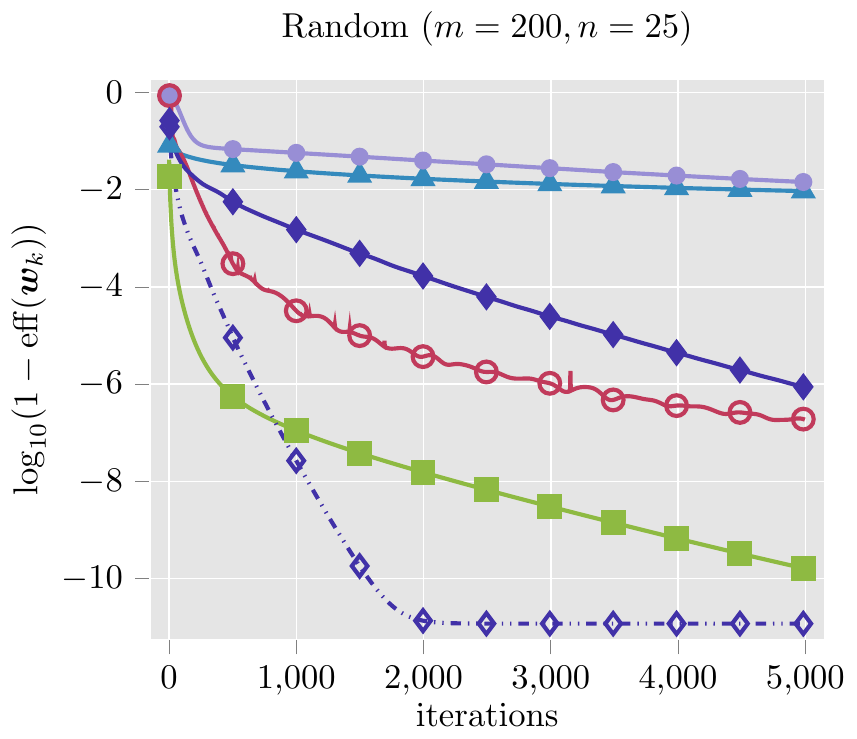}} & \hspace{-5mm}\scalebox{0.37}{\includegraphics[clip,trim=8mm 0 0 0]{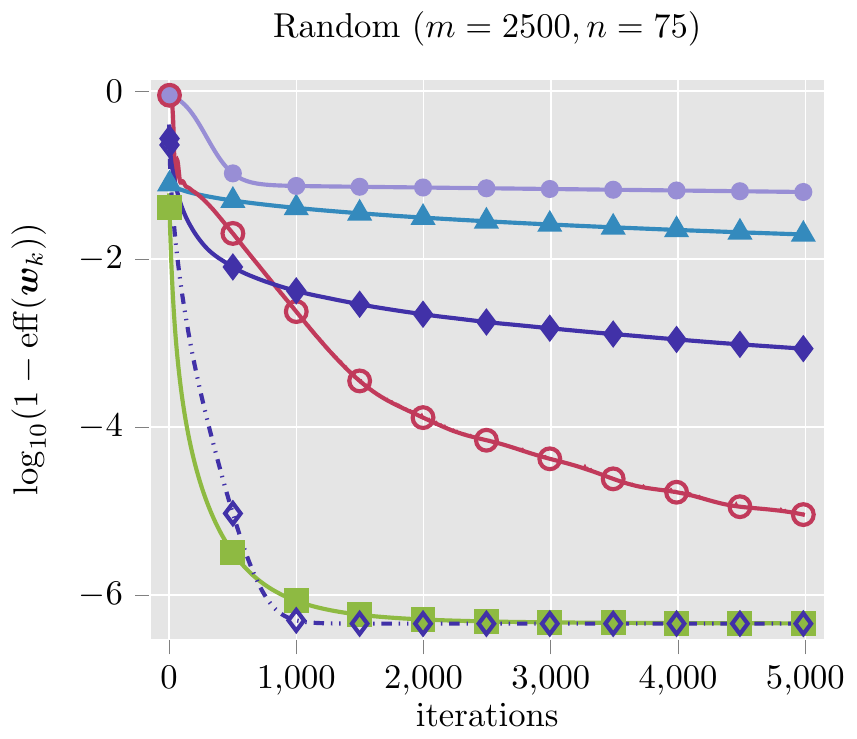}} & \hspace{-5mm}\scalebox{0.37}{\includegraphics[clip,trim=8mm 0 0 0]{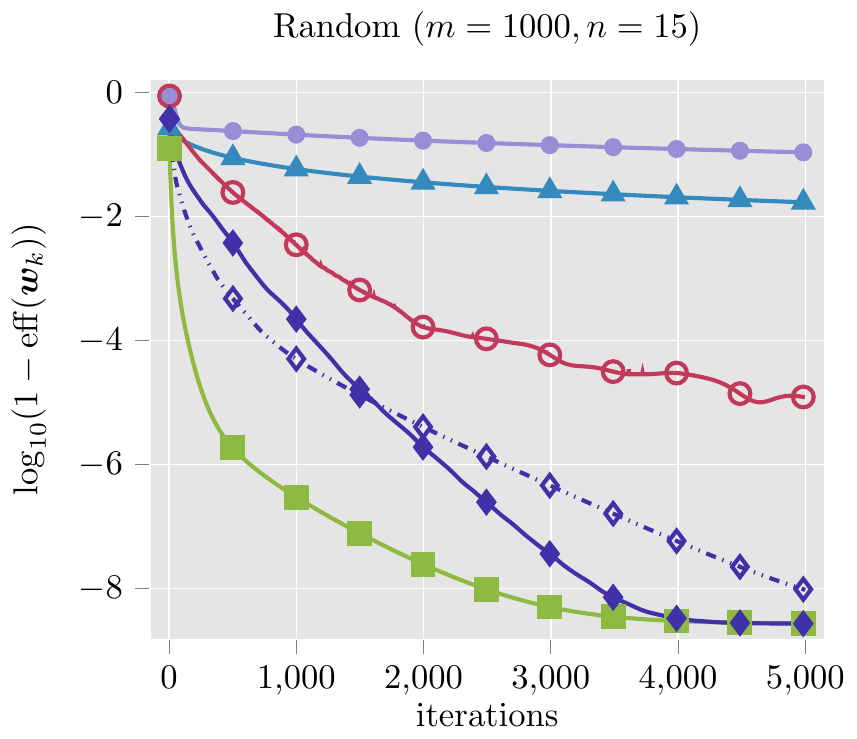}} & \hspace{-5mm}\scalebox{0.37}{\includegraphics[clip,trim=8mm 0 0 0]{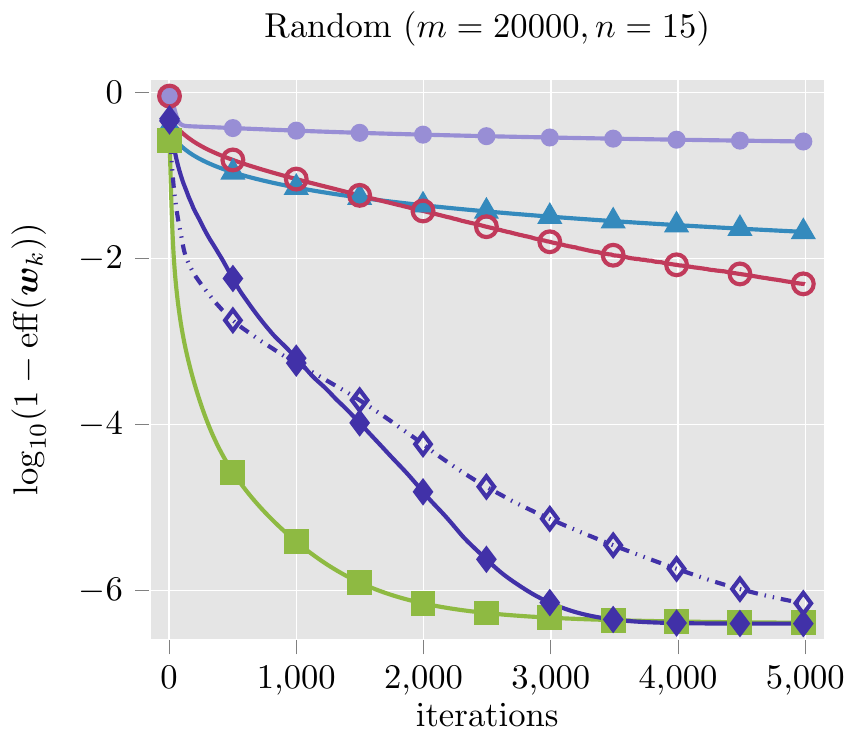}}\\
   \hspace{-3mm} \scalebox{0.37}{\includegraphics{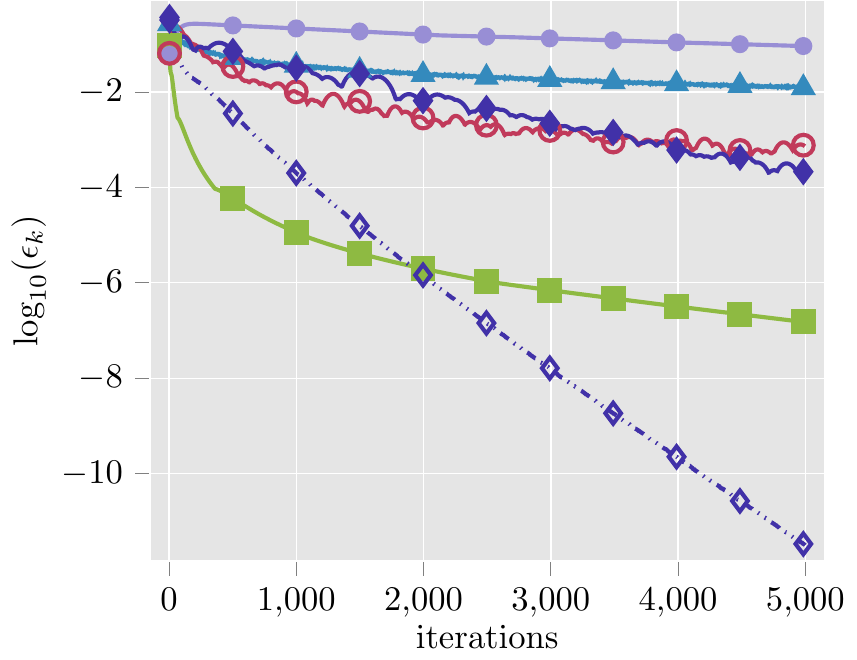}} & \hspace{-5mm}\scalebox{0.37}{\includegraphics[clip,trim=8mm 0 0 0]{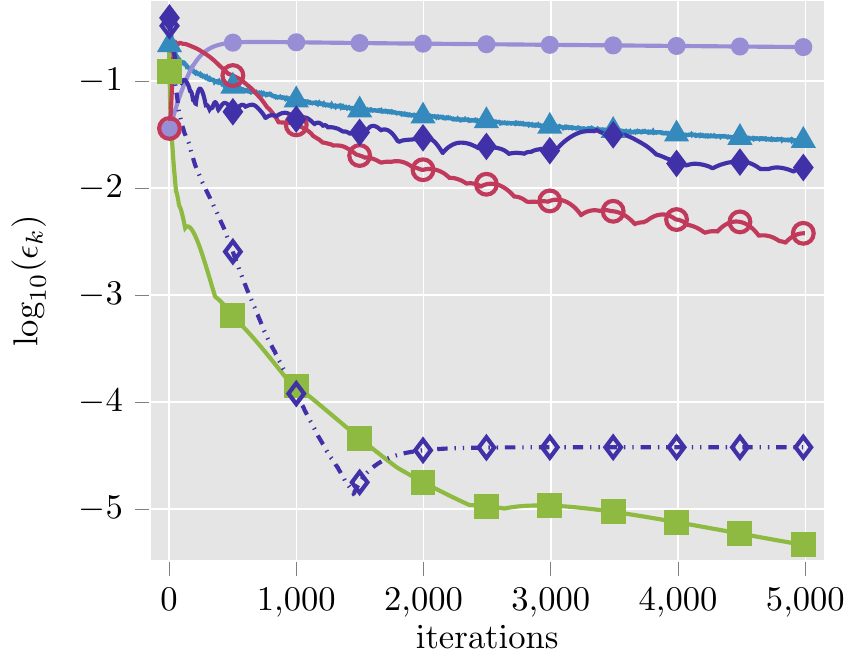}} & \hspace{-5mm}\scalebox{0.37}{\includegraphics[clip,trim=8mm 0 0 0]{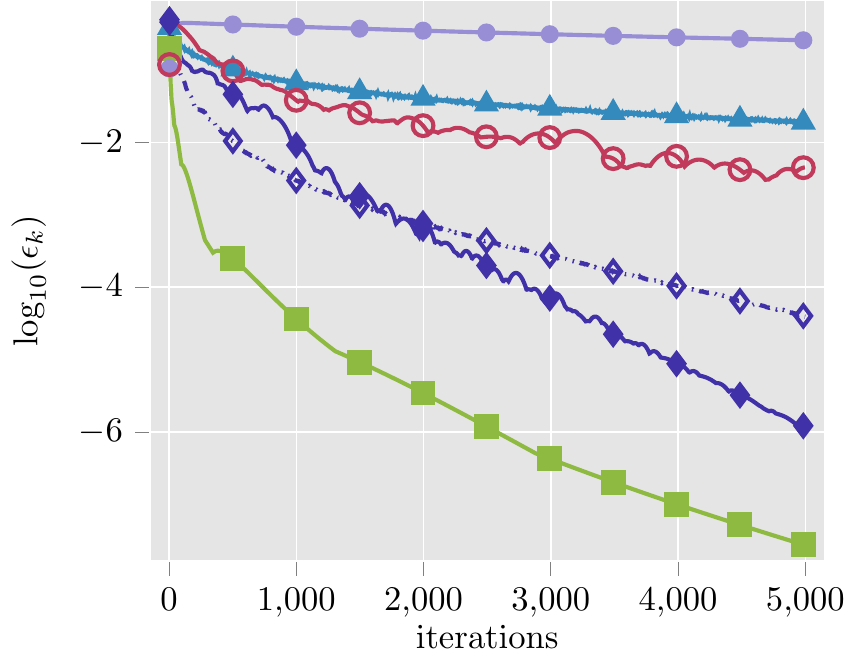}} & \hspace{-5mm}\scalebox{0.37}{\includegraphics[clip,trim=8mm 0 0 0]{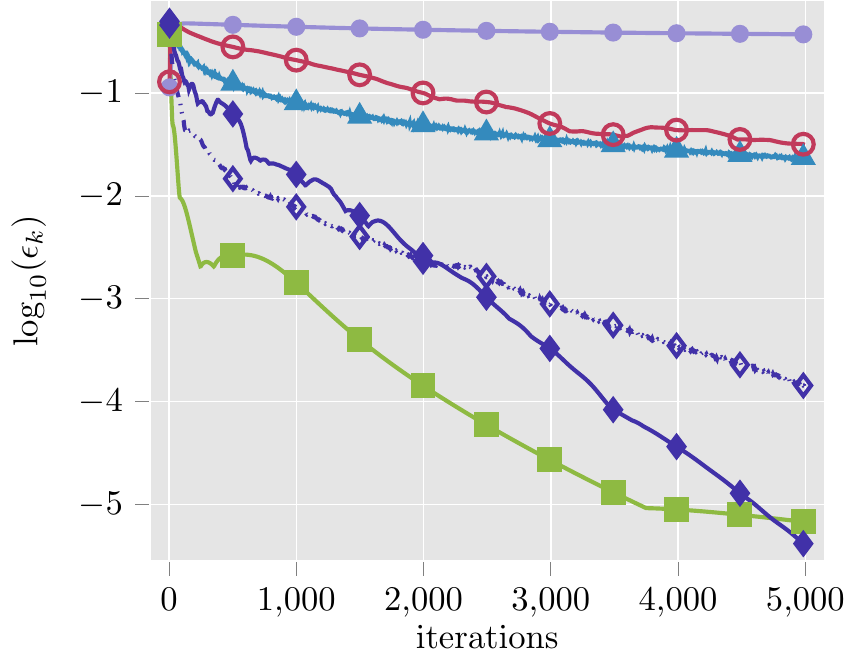}}\\
   \hspace{-3mm} \scalebox{0.37}{\includegraphics{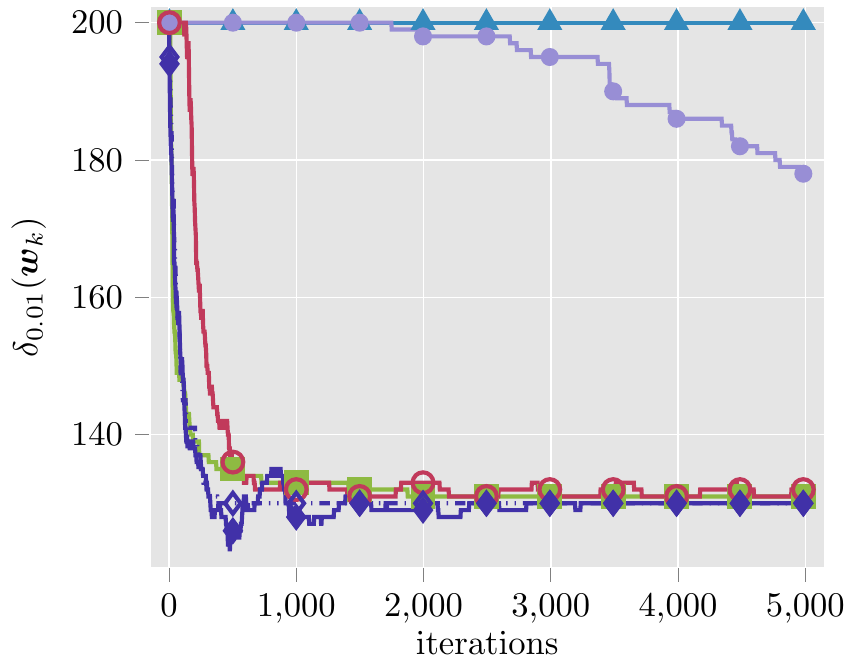}} & \hspace{-5mm}\scalebox{0.37}{\includegraphics[clip,trim=4.6mm 0 0 0]{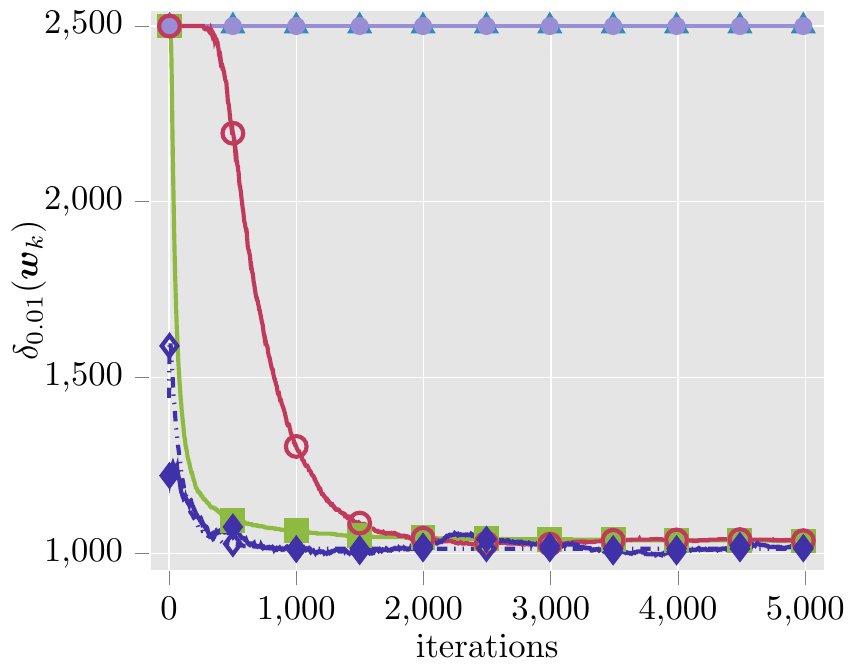}} & \hspace{-5mm}\scalebox{0.37}{\includegraphics[clip,trim=4.6mm 0 0 0]{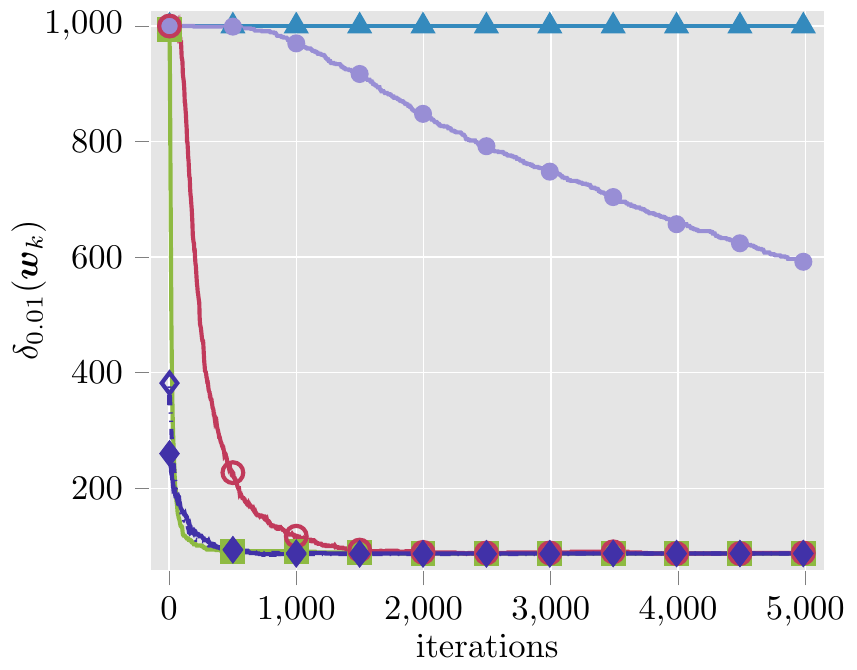}} & \hspace{-5mm}\scalebox{0.37}{\includegraphics[clip,trim=8mm 0 0 0]{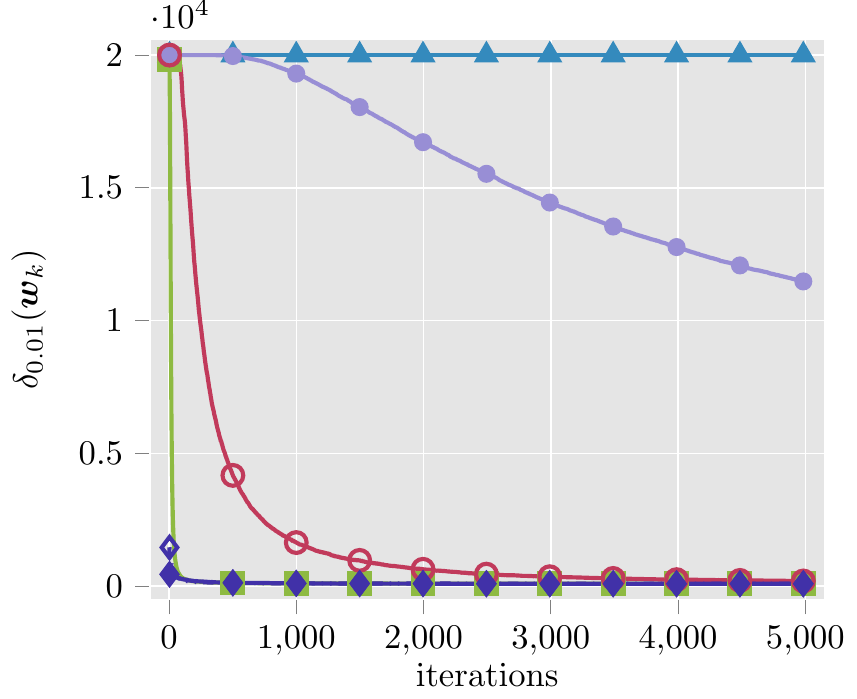}}
 \end{tabular}
 \begin{center}
 \includegraphics[height=1.1em]{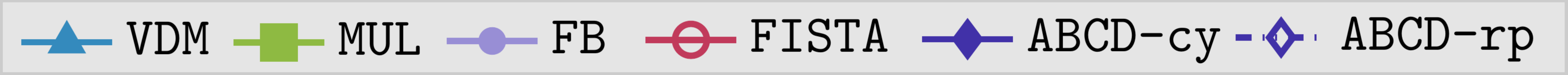}  
 \end{center}
 \caption{Efficiency, duality bound and support size for four random optimal design instances. \MUL\ and alternating minimization methods are the best performers.} \label{fig:results-r}
\end{figure}

 \begin{figure}[t]
 \begin{tabular}{cccl}
   \hspace{-3mm} \scalebox{0.37}{\includegraphics{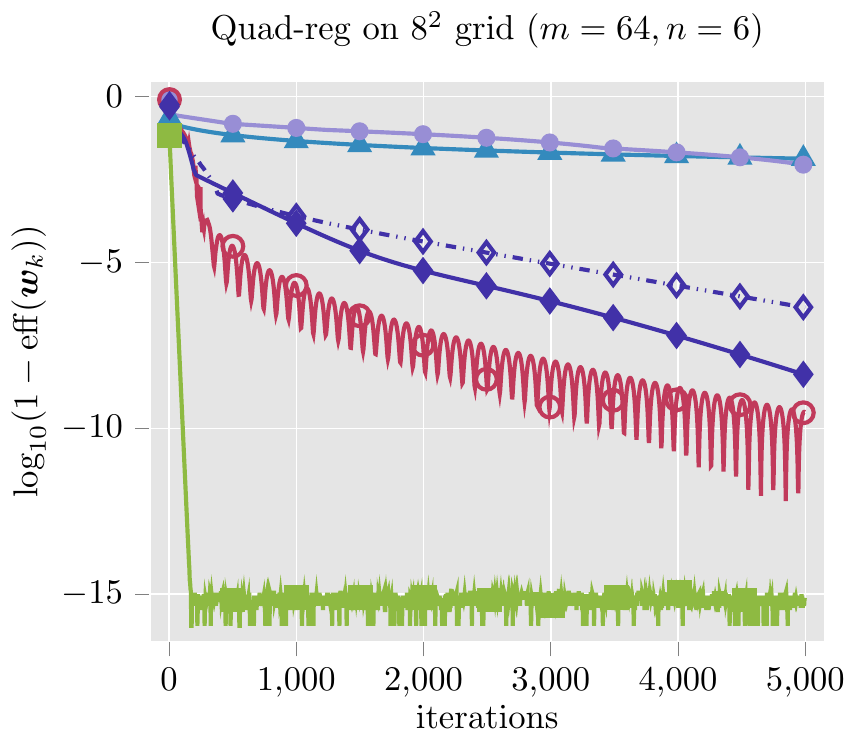}} & \hspace{-5mm}\scalebox{0.37}{\includegraphics[clip,trim=8mm 0 0 0]{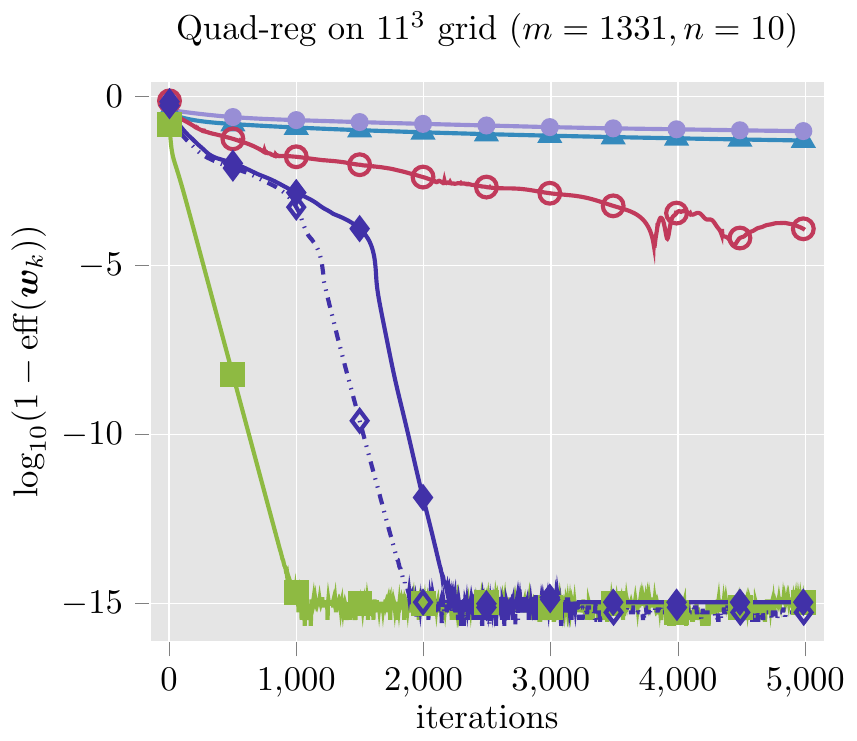}} & \hspace{-5mm}\scalebox{0.37}{\includegraphics[clip,trim=8mm 0 0 0]{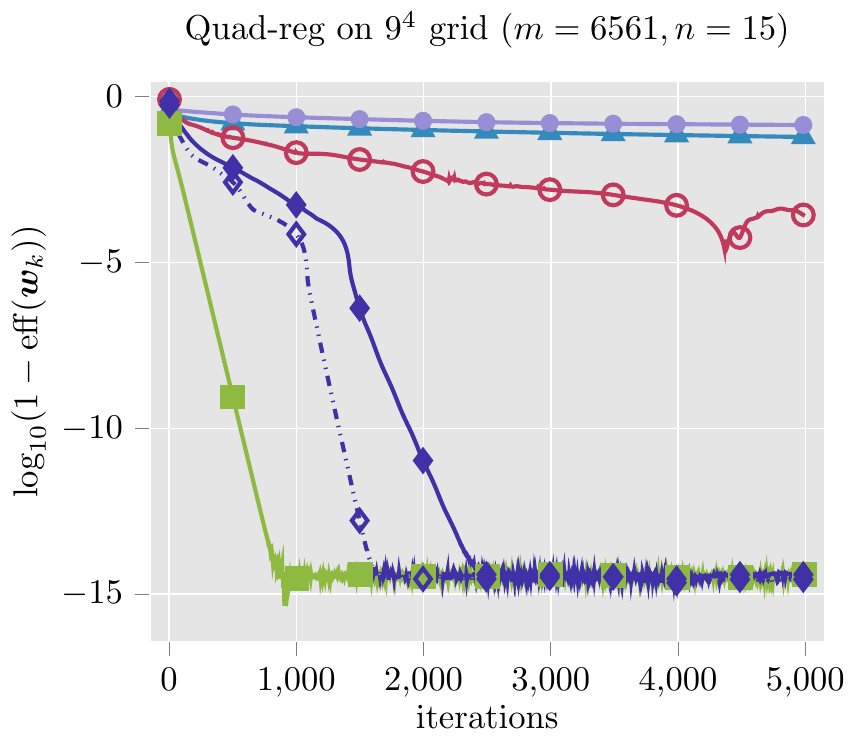}} & \hspace{-5mm}\scalebox{0.37}{\includegraphics[clip,trim=8mm 0 0 0]{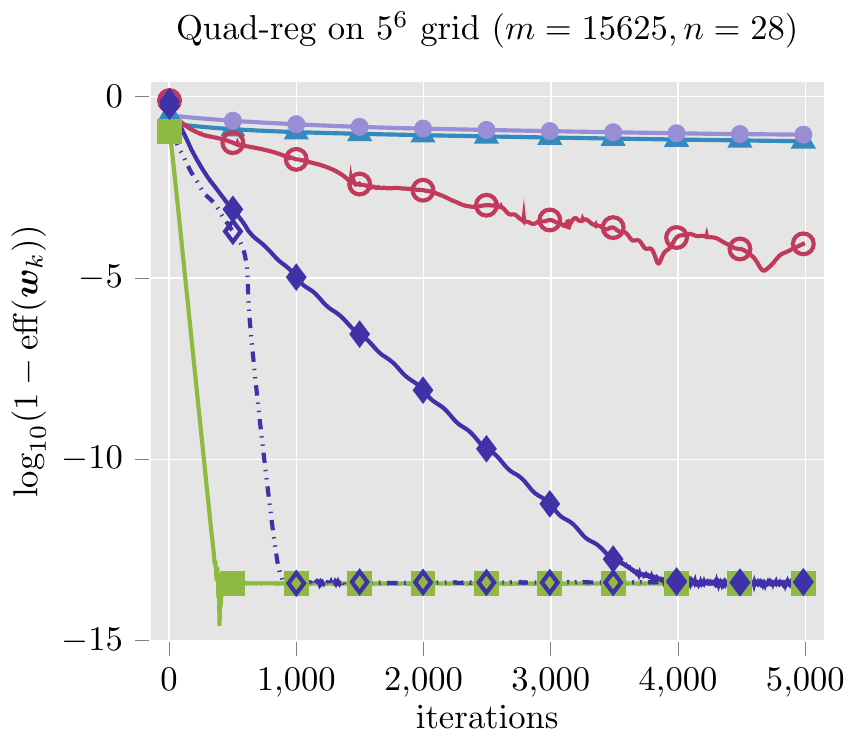}}\\
   \hspace{-3mm} \scalebox{0.37}{\includegraphics{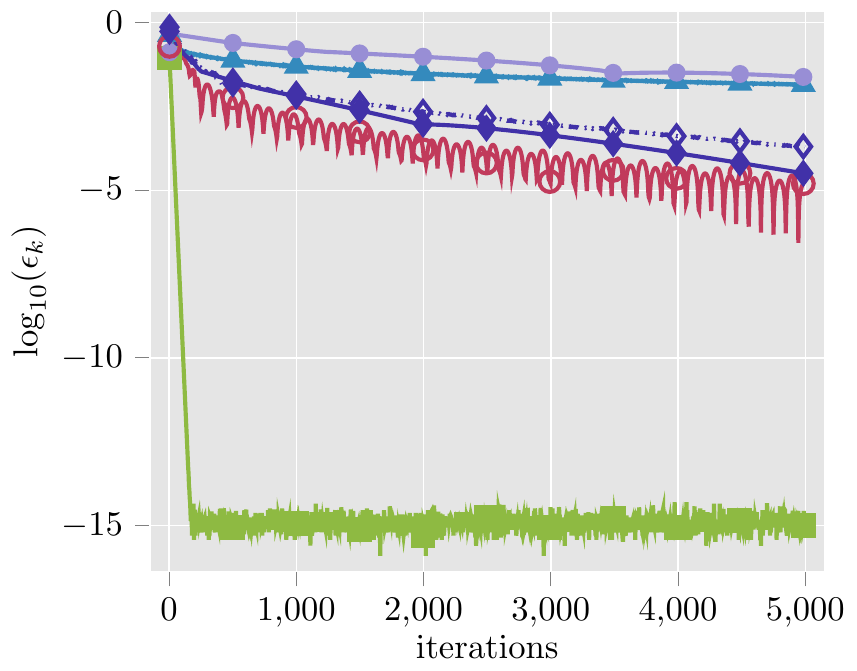}} & \hspace{-5mm}\scalebox{0.37}{\includegraphics[clip,trim=8mm 0 0 0]{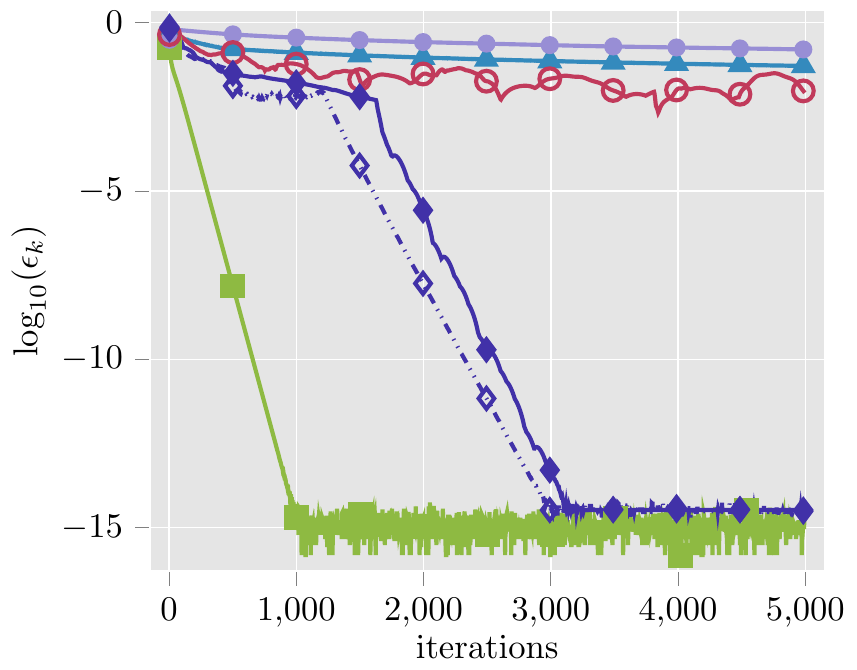}} & \hspace{-5mm}\scalebox{0.37}{\includegraphics[clip,trim=8mm 0 0 0]{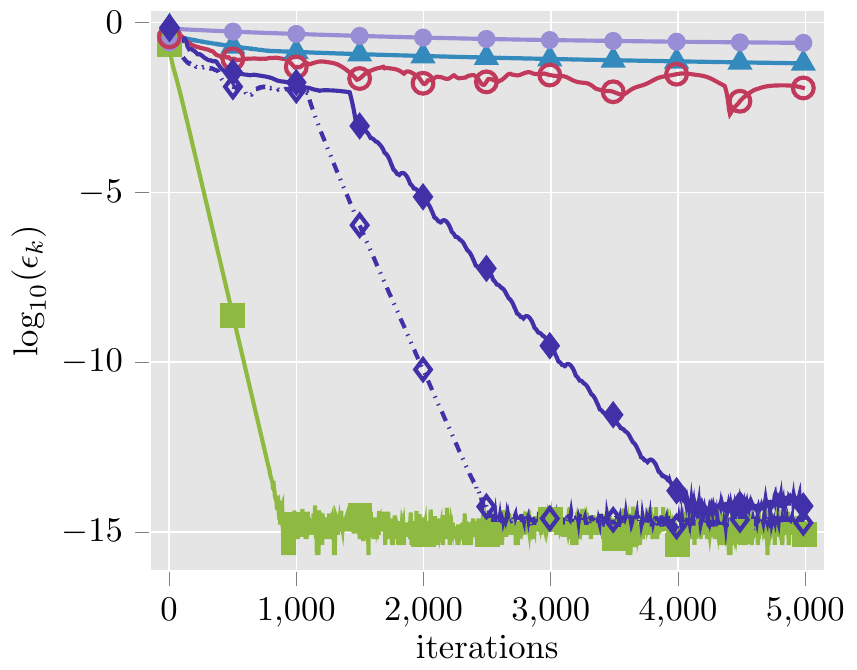}} & \hspace{-5mm}\scalebox{0.37}{\includegraphics[clip,trim=8mm 0 0 0]{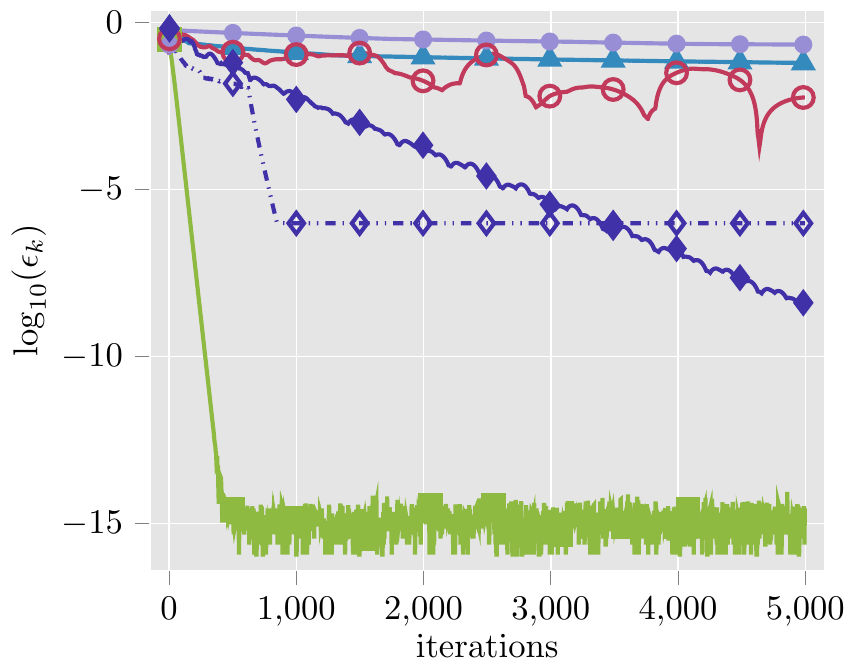}}\\
   \hspace{-3mm} \scalebox{0.37}{\includegraphics{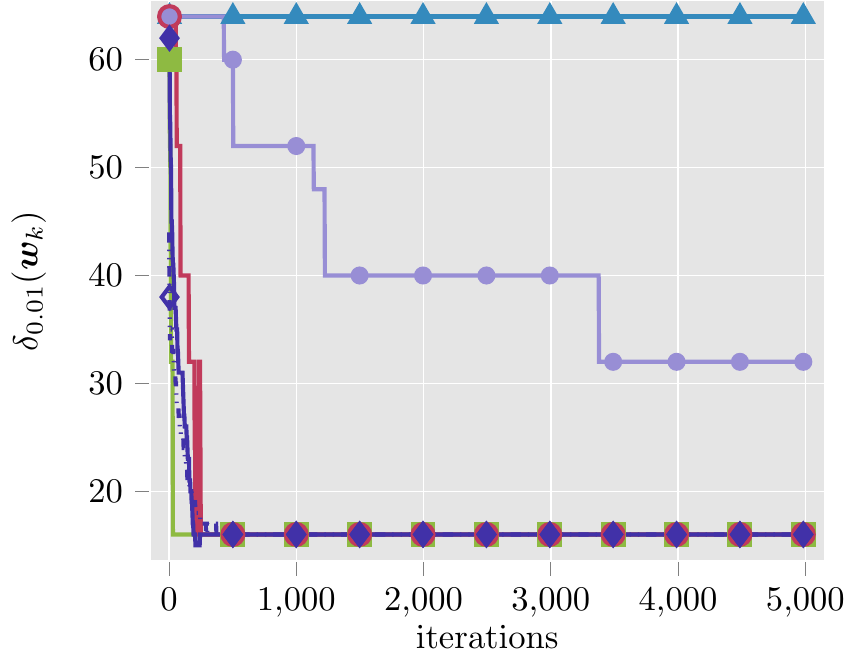}} & \hspace{-5mm}\scalebox{0.37}{\includegraphics[clip,trim=4.6mm 0 0 0]{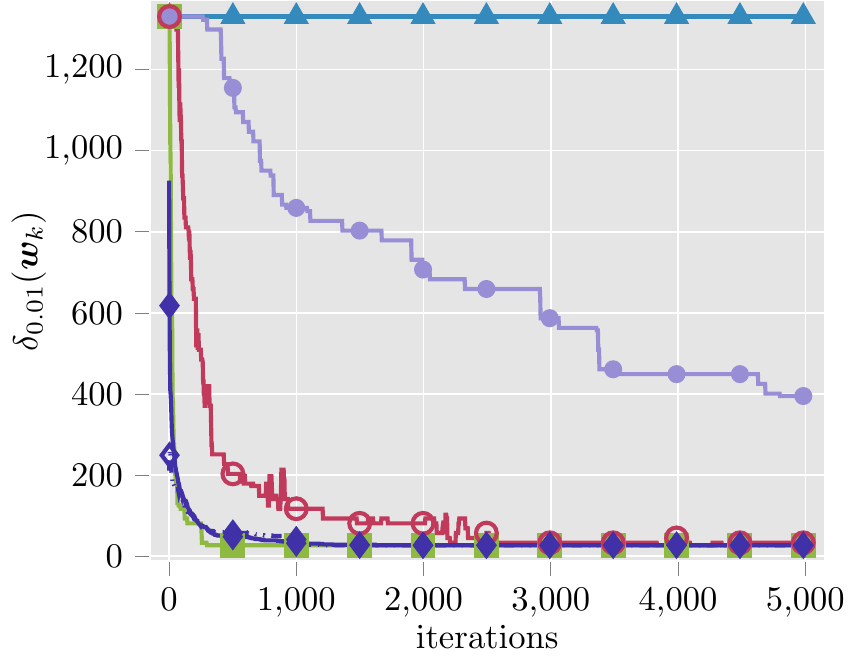}} & \hspace{-5mm}\scalebox{0.37}{\includegraphics[clip,trim=4.6mm 0 0 0]{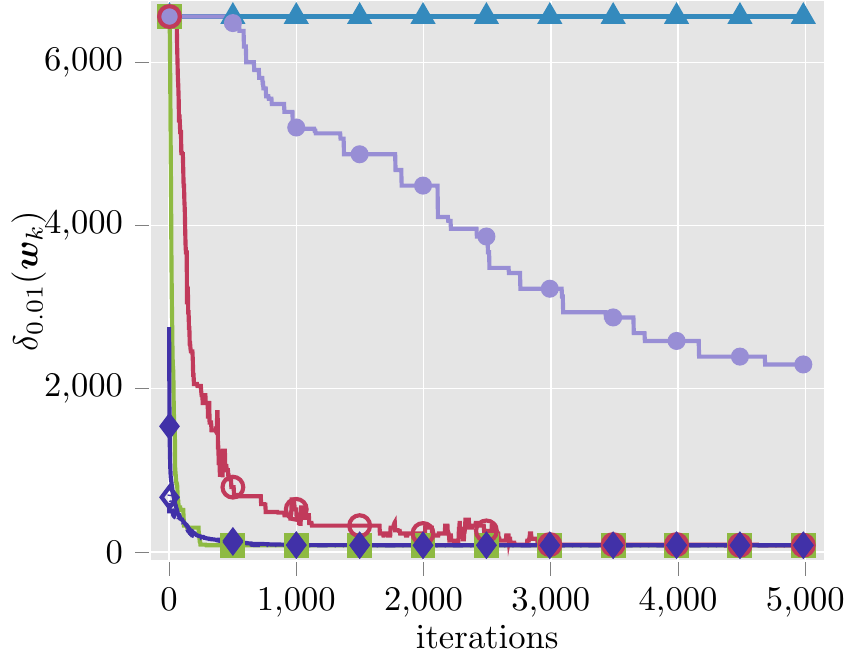}} & \hspace{-5mm}\scalebox{0.37}{\includegraphics[clip,trim=8mm 0 0 0]{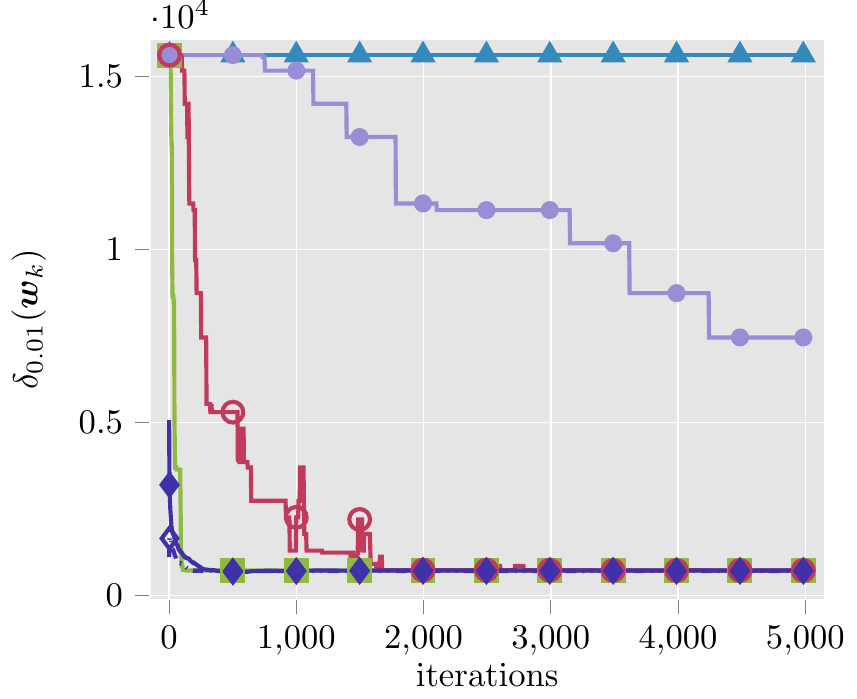}}
 \end{tabular}
 \begin{center}
 \includegraphics[height=1.1em]{legend.pdf}  
 \end{center}
 \caption{Efficiency, duality bound and support size for four quadratic regression instances. \MUL\ and alternating minimization methods are the best performers. Flat lines indicate that the algorithms have reached numerical precision.} \label{fig:results-q}
\end{figure}

\medskip
Another important measure of a design's quality is its sparsity. It is well known that optimal designs are supported by a few points only,
which is a desired property for many applications.
The problem formulation~\eqref{eq:SGL} gives a new explanation for this fact, as the penalty term $\Omega(X)$
is added in group lasso regression in order to induce block sparsity,
so we expect the optimal matrix $X^*$ to have a lot of columns equal to $\vec{0}$
(the squared penalty term $g(X)=\frac{1}{2} \Omega(X)^2$ is also known to be block-sparsity inducing, cf.~\cite{BAC08}).

A remarkable property of the proximal decomposition methods presented in this article is that
$\operatorname{prox}_{tg}(V)$ acts as a thresholding operator on $V$, literally zeroing a lot of columns.
The same is true for alternating block coordinate descent methods, in which whole columns are set to $\vec{0}$ if a certain threshold property holds.
As a result, the iterates produced by \FB, \FISTA, \ABCDcy\ and \ABCDrp\ are expected to have a small support. 
To observe this fact, we measure the sparsity of a design by $\delta_{0.01}(\vec{w}_k)$,
the number of coordinates of $\ww_k$ exceeding the value $\frac{0.01}{m}$.

\bigskip
The evolution of the efficiency, the duality bound $1-\varepsilon_k$, and the support size
during the 5000 first iterations of each algorithm is depicted
in Figure~\ref{fig:results-r} and Figure~\ref{fig:results-q}
for four random and quadratic regression instances of various sizes. 
As already mentioned,
all the algorithms we compare have a complexity of $O(n^2m)$ per iteration. It is therefore possible to get a rough idea of their comparative computational efficiency from an \emph{iteration-based} performance analysis.
Performing a more precise time-based analysis depends on optimization of the linear algebra operations required for each algorithm and is beyond the scope of this work, so we stick to iteration-based analysis.

We observe several properties of the new algorithms on these figures. First, the effect of acceleration can clearly be seen on the figures, as \FISTA\ always beats \VDM,
while the simple forward-backward algorithm \FB\ is typically outperformed. As explained in Section~\ref{sec:prox_dec}, this comes at the price of \FISTA\ \emph{not being} a descent method,
which can also be observed on the plots, especially for the quadratic regression instances (Figure~\ref{fig:results-q}).
Second, \MUL\ is performing in general better than other algorithms, closely followed by alternating minimization methods. The performance of the remaining algorithm is in general bellow. 
These preliminary results suggest that the group lasso formulation of the experimental design problem has the potential to help deciphering powerful algorithms for the later problem.  In particular, the alternating block coordinate descents exhibit a nice linear convergence on many instances.
Pushing further would requires to look more carefully at the implementation details of each algorithms and perform much larger scale experiments which is beyond the scope of this paper. Many upgrades and improvements also have to be evaluated, e.g.\ preconditioning, clever subsampling of the blocks to be updated at each iteration.
Third, the support plots show that \FISTA\ and \texttt{ABCD} quickly converge to a sparse solution. \MUL\ also quickly identifies a design with a small support. We recall that the iterates of \FISTA\ and \texttt{ABCD} are \emph{truly sparse}, 
while for \MUL, this is only a \emph{numerical sparsity}, as the iterates $\ww_k$ remain strictly positive.
On the other hand \VDM\ always fail to identify a smaller support.
Finite time identification of sparsity patterns is an active topic of research in nonsmooth optimization and we believe that this property can be exploited to yield high-performance algorithms to solve very large scale optimal design problems.

\section{Conclusion}

This paper presents a strong, previously unrevealed connection between two standard problems in statistics (Bayes A-optimal design and group lasso regression),
hence clearing the path to a convergence of algorithms used in the communities of optimal design of experiments and machine learning.
While the new methods presented in this article are not yet competitive with other algorithms for computing optimal designs over a finite design space,
they certainly present interesting features, such as sparse iterates and a guaranteed speed of convergence,
and we believe that there is still an important room for improvement, e.g. by using
recent techniques based on subsampling oracles~\cite{KPA18} or lazy separators~\cite{BPZ17}.
Conversely, an interesting perspective is to use well established techniques of optimal experimental design, such as methods to restrict the set of potential support points
of an optimal design~\cite{Pro13}, to improve
algorithms that were designed to solve group lasso regressions. 

Another topic for further research is whether we can reformulate other design problems (such as the D-optimal design problem, or problems with constraints on the design weights)
as unconstrained convex optimization problems.




\begin{thebibliography}{10}

\bibitem{AD92}
A.C. Atkinson and A.N. Donev.
\newblock {\em Optimum Experimental Designs}, volume~8.
\newblock Oxford Statistical Science Series, 1992.

\bibitem{BAC08}
F.R. Bach.
\newblock Consistency of the group lasso and multiple kernel learning.
\newblock {\em Journal of Machine Learning Research}, 9(Jun):1179--1225, 2008.

\bibitem{bach2012optimization}
Francis Bach, Rodolphe Jenatton, Julien Mairal, Guillaume Obozinski, et~al.
\newblock Optimization with sparsity-inducing penalties.
\newblock {\em Foundations and Trends{\textregistered} in Machine Learning},
  4(1):1--106, 2012.

\bibitem{beck2009fast}
Amir Beck and Marc Teboulle.
\newblock A fast iterative shrinkage-thresholding algorithm for linear inverse
  problems.
\newblock {\em SIAM journal on imaging sciences}, 2(1):183--202, 2009.

\bibitem{bertsekas1999nonlinear}
Dimitri~P Bertsekas.
\newblock {\em Nonlinear programming}.
\newblock Athena scientific Belmont, 2nd edition, 1999.

\bibitem{Boh86}
D.~B\"{o}hning.
\newblock A vertex-exchange-method in {D}-optimal design theory.
\newblock {\em Metrika}, 33(1):337--347, 1986.

\bibitem{borwein2010convex}
Jonathan Borwein and Adrian~S Lewis.
\newblock {\em Convex analysis and nonlinear optimization: theory and
  examples}.
\newblock Springer Science \& Business Media, 2010.

\bibitem{BV04}
S.~Boyd and L.~Vandenberghe.
\newblock {\em Convex Optimization}.
\newblock Cambridge University Press, 2004.

\bibitem{BPZ17}
G.~Braun, S.~Pokutta, and D.~Zink.
\newblock Lazifying conditional gradient algorithms.
\newblock In {\em International Conference on Machine Learning}, pages
  566--575, 2017.

\bibitem{CH12}
M.~{\v{C}}ern{\`y} and M.~Hlad{\'\i}k.
\newblock Two complexity results on c-optimality in experimental design.
\newblock {\em Computational Optimization and Applications}, 51(3):1397--1408,
  2012.

\bibitem{Cha84}
K.~Chaloner.
\newblock Optimal bayesian experimental design for linear models.
\newblock {\em The Annals of Statistics}, pages 283--300, 1984.

\bibitem{chambolle2015convergence}
Antonin Chambolle and Charles Dossal.
\newblock {On the convergence of the iterates of" FISTA"}.
\newblock {\em Journal of Optimization Theory and Applications}, 166(3):25,
  2015.

\bibitem{combettes2011proximal}
Patrick~L Combettes and Jean-Christophe Pesquet.
\newblock Proximal splitting methods in signal processing.
\newblock In {\em Fixed-point algorithms for inverse problems in science and
  engineering}, pages 185--212. Springer, 2011.

\bibitem{DD76}
G.~Duncan and M.H. De{G}root.
\newblock A mean squared error approach to optimal design theory.
\newblock In {\em Proceedings of the 1976 Conference on Information: Science
  and systems}, pages 217--221, The John Hopkins University, 1976.

\bibitem{FL00}
V.~Fedorov and J.~Lee.
\newblock Design of experiments in statistics.
\newblock In H.Wolkowicz, R.Saigal, and L.Vandenberghe, editors, {\em Handbook
  of semidefinite programming}, chapter~17. Kluwer, 2000.

\bibitem{Fed72}
V.V. Fedorov.
\newblock {\em Theory of optimal experiments}.
\newblock New York : Academic Press, 1972.
\newblock Translated and edited by W. J. Studden and E. M. Klimko.

\bibitem{FW56}
M.~Frank and P.~Wolfe.
\newblock An algorithm for quadratic programming.
\newblock {\em Naval research logistics quarterly}, 3(1-2):95--110, 1956.

\bibitem{gauthier2016optimal}
B.~Gauthier and L.~Pronzato.
\newblock Optimal design for prediction in random field models via covariance
  kernel expansions.
\newblock In {\em mODa 11-Advances in Model-Oriented Design and Analysis},
  pages 103--111. Springer, 2016.

\bibitem{gauthier2017convex}
B.~Gauthier and L.~Pronzato.
\newblock Convex relaxation for imse optimal design in random-field models.
\newblock {\em Computational Statistics \& Data Analysis}, 113:375--394, 2017.

\bibitem{HFR18}
R.~Harman, L.~Filov\'{a}, and P.~Richt\'{a}rik.
\newblock A randomized exchange algorithm for computing optimal approximate
  designs of experiments.
\newblock arXiv preprint 1801.05661, 2018.

\bibitem{KPA18}
T.~Kerdreux, F.~Pedregosa, and A.~d'Aspremont.
\newblock Frank-wolfe with subsampling oracle.
\newblock {\em arXiv preprint arXiv:1803.07348}, 2018.

\bibitem{LMW18}
J.~Lukemire, A.~Mandal, and W.K. Wong.
\newblock {D-QPSO}: A quantum-behaved particle swarm technique for finding
  d-optimal designs with discrete and continuous factors and a binary response.
\newblock {\em Technometrics}, pages 1--27, 2018.
\newblock e-pub ahead of print.

\bibitem{nesterov1983method}
Yurii Nesterov.
\newblock {A method of solving a convex programming problem with convergence
  rate O (1/k2)}.
\newblock In {\em Soviet Mathematics Doklady}, volume~27, pages 372--376, 1983.

\bibitem{Pro13}
L.~Pronzato.
\newblock A delimitation of the support of optimal designs for kiefer’s
  $\phi$p-class of criteria.
\newblock {\em Statistics \& Probability Letters}, 83(12):2721--2728, 2013.

\bibitem{Puk93}
F.~Pukelsheim.
\newblock {\em Optimal Design of Experiments}.
\newblock Wiley, 1993.

\bibitem{PR92}
F.~Pukelsheim and S.~Rieder.
\newblock Efficient rounding of approximate designs.
\newblock {\em Biometrika}, pages 763--770, 1992.

\bibitem{rockafellar1970convex}
R.~T. Rockafellar.
\newblock {\em Convex analysis}.
\newblock Princeton Mathematical Series, No. 28. Princeton University Press,
  Princeton, N.J., 1970.

\bibitem{rockafellar2009variational}
R~Tyrrell Rockafellar and Roger J-B Wets.
\newblock {\em Variational analysis}, volume 317.
\newblock Springer Science \& Business Media, 2009.

\bibitem{SH15}
G.~Sagnol and R.~Harman.
\newblock Computing exact {D}-optimal designs by mixed integer second-order
  cone programming.
\newblock {\em The Annals of Statistics}, 43(5):2198--2224, 2015.

\bibitem{SagnolHegeWeiser2016}
G.~Sagnol, H.-C. Hege, and M.~Weiser.
\newblock Using sparse kernels to design computer experiments with tunable
  precision.
\newblock In {\em Proceedings of the 22nd International Conference on
  Computational Statistics}, pages 397--408, 2016.

\bibitem{Sag11}
Guillaume Sagnol.
\newblock Computing optimal designs of multiresponse experiments reduces to
  second-order cone programming.
\newblock {\em Journal of Statistical Planning and Inference},
  141(5):1684--1708, 2011.

\bibitem{STT78}
S.D. Silvey, D.M. Titterington, and B.~Torsney.
\newblock An algorithm for optimal designs on a finite design space.
\newblock {\em Communications in Statistics - Theory and Methods},
  7(14):1379--1389, 1978.

\bibitem{TM13}
K.~Tanaka and M.~Miyakawa.
\newblock The group lasso for design of experiments.
\newblock {\em arXiv preprint arXiv:1308.1196}, 2013.

\bibitem{wright2015coordinate}
Stephen~J Wright.
\newblock Coordinate descent algorithms.
\newblock {\em Mathematical Programming}, 151(1):3--34, 2015.

\bibitem{Wyn70}
H.P. Wynn.
\newblock The sequential generation of {$D$}-optimum experimental designs.
\newblock {\em Annals of Mathematical Statistics}, 41:1655--1664, 1970.

\bibitem{Yu10a}
Y.~Yu.
\newblock Monotonic convergence of a general algorithm for computing optimal
  designs.
\newblock {\em The Annals of Statistics}, 38(3):1593--1606, 2010.

\bibitem{Yu11}
Y.~Yu.
\newblock {D}-optimal designs via a cocktail algorithm.
\newblock {\em Statistics and Computing}, 21(4):475--481, 2011.

\bibitem{yuan2006model}
Ming Yuan and Yi~Lin.
\newblock Model selection and estimation in regression with grouped variables.
\newblock {\em Journal of the Royal Statistical Society: Series B (Statistical
  Methodology)}, 68(1):49--67, 2006.

\end{thebibliography}


\end{document}